\documentclass[12pt]{amsart}

\rmfamily{\fontsize{12pt}{\baselineskip}\selectfont}

\usepackage[bbgreekl]{mathbbol}
\usepackage{mathrsfs}
\usepackage{graphicx}
\usepackage{amsmath}
\usepackage{amsfonts}
\usepackage{indentfirst}
\usepackage{amsthm}
\usepackage{amssymb}
\usepackage{algorithm}
\usepackage{algorithmic}
\usepackage{xcolor}
\usepackage{geometry}
\usepackage{appendix}
\usepackage{bm}
\usepackage{url}
\usepackage{setspace}
\usepackage{subfig}
\usepackage{verbatim}
\usepackage[normalem]{ulem}
\usepackage{multirow}
\usepackage{lineno}
\usepackage{float}
\usepackage{array}
\usepackage{caption}
\usepackage{diagbox}
\usepackage{hyperref}

\newtheorem{theorem}{Theorem}[section]
\newtheorem{lemma}{Lemma}[section]
\newtheorem{remark}{Remark}[section]

\newtheorem{corollary}{Corollary}[section]

\definecolor{yscol}{HTML}{6622AA}

\definecolor{jbcol}{HTML}{A51535}

\definecolor{lzcol}{HTML}{0000ff}

\date{\today}

\begin{document}

\title[Higher-order Far-field Boundary Conditions]{Higher-order Far-field Boundary Conditions \\
for Crystalline Defects}

\author{Julian Braun}
\address{Heriot-Watt University, Edinburgh, EH14 4AS, UK}
\email{j.braun@hw.ac.uk}
\author{Christoph Ortner}
\address{University of British Columbia, 1984 Mathematics Road, Vancouver, BC, Canada.}
\email{ortner@math.ubc.ca}
\author{Yangshuai Wang}
\address{University of British Columbia, 1984 Mathematics Road, Vancouver, BC, Canada.}
\email{yswang2021@math.ubc.ca}
\author{Lei Zhang}
\address{School of Mathematical Sciences, Institute of Natural Sciences and MOE-LSC, Shanghai Jiao Tong University, Shanghai 200240, China.}
\email{lzhang2012@sjtu.edu.cn}


\begin{abstract}
Crystalline materials exhibit long-range elastic fields due to the presence of defects, leading to significant domain size effects in atomistic simulations. 
A rigorous far-field expansion of these long-range fields identifies low-rank structure in the form of a sum of discrete multipole terms and continuum correctors~\cite{2021-defectexpansion}. We propose a novel numerical scheme that exploits this low-rank structure to accelerate material defect simulations by minimizing the domain size effects.
Our approach iteratively improves the boundary condition, systematically following the asymptotic expansion of the far field. We provide both rigorous error estimates for the method and a range of empirical numerical tests, to assess it's convergence and robustness. 
%
\end{abstract}

\maketitle


\renewcommand\arraystretch{1.5}


\newcommand{\G}{\mathcal{G}}

\def\<{\langle}
\def\>{\rangle}

\newcommand{\R}{\mathbb{R}}

\newcommand{\Rc}{\ensuremath{\mathcal{R}}}
\newcommand{\Rcc}{\ensuremath{\overline{\mathcal{R}}}}
\newcommand{\La}{\Lambda}

\def\HH{\mathcal{H}^{1}}
\def\HHc{\mathcal{H}^{\rm c}}
\def\Ladef{\La^{\rm def}}
\def\Rdef{R^{\rm def}}

\def\CC{{\bf C}}
\def\UU{{\bf U}}
\def\Gr{\mathcal{G}}

\def\mA{\mathsf{A}}
\def\Div{{\rm Div}}
\newcommand{\Q}{\mathbb{Q}}
\newcommand{\Z}{\mathbb{Z}}
\newcommand{\C}{\mathbb{C}}
\newcommand{\N}{\mathbb{N}}
\newcommand{\M}{\mathbb{M}}
\newcommand{\wto}{\rightharpoonup}
\newcommand{\wsto}{\stackrel{*}{\rightharpoonup}}
\newcommand{\css}{\subset\subset}
\newcommand{\eps}{\varepsilon}

\def\dt{\,{\textrm{d}}t}
\def\ds{\,{\textrm{d}}s}
\def\dsigma{\,{\textrm{d}}\sigma}

\section{Introduction}
\label{sec:intro}
This work is concerned with the precise characterization of geometric and energetic properties of defects within crystalline materials~\cite{groger2008multiscale, horstemeyer2009multiscale, steinhauser2017computational, yip2007handbook}. 
Practical simulation schemes operate within finite domains and hence cannot fully resolve the long-ranged elastic far-field, hence the computation of these properties requires the careful consideration of artificial (e.g., clamped or periodic) boundary conditions~\cite{2018-uniform, EOS2016}. 
The choice of boundary condition significantly influences the emergence of cell-size effects, an issue that will be studied in detail in this paper. 

Standard supercell methods for modeling defect equilibration are recognized for their relatively slow convergence concerning cell size~\cite{chen19, EOS2016}. Addressing this limitation systematically necessitates the development of {\em higher-order boundary conditions}, aiming to improve the  convergence rates in terms of cell size. An important step in this endeavor is the careful analysis of the elastic far-fields induced by the presence of defects.

A common approach to characterize the elastic far-field behavior is to leverage the low-rank structure of defect configurations. This involves modeling defects using continuum linear elasticity and the defect dipole tensor, first introduced in \cite{eshelby1956continuum, nowick1963anelasticity}. This concept was later applied to atomistic models of fracture \cite{sinclair1975influence, sinclair1972atomistic}. More recent related works utilize lattice Green's functions to improve accuracy in defect computations~\cite{tan2016computation, trinkle2008lattice}. However, these contributions tend to be application-focused, lacking a rigorous mathematical foundation and framework for systematically designing and improving the models and numerical algorithms.

Recently, Braun {\it et al.}~\cite{2017-bcscrew, 2021-defectexpansion} developed a unified mathematical framework that exploits the low-rank defect structure to characterize the elastic far-fields. In this framework, the defect equilibrium is decomposed into a sum of continuum correctors and discrete multipole terms. This novel formulation exposes avenues for improved convergence rates in cell problems concerning cell size, as demonstrated theoretically. However, a notable challenge arises in the practical implementation of multipole expansions for simulating crystalline defects, given that the terms associated with multipole moments are defined on an infinite lattice, rendering their direct evaluation unfeasible within finite computational domains. 

The purpose of the present paper is to propose a novel numerical framework that utilizes the multipole expansions of Braun {\it et al.}~\cite{2017-bcscrew, 2021-defectexpansion} to accelerate the simulation of crystalline defects. 
We design an iterative method that systematically enhances the accuracy of approximate multipole tensor evaluations, accompanied by rigorous error estimates.
Additionally, we leverage a continuous version of multipole expansions, employing continuous Green's functions to further enhance computational efficiency. To evaluate the effectiveness of our approach, we present numerical examples for a range of point defects, assessing both the geometry error and energy error convergence. Our numerical results demonstrate that the proposed framework for higher-order boundary conditions effectively achieves accelerated convergence rates with respect to computational cell size.

This paper concentrates on point defects to provide a comprehensive understanding of key concepts. Future research will delve into broader applications of this approach. However, extending the method to complex scenarios like edge dislocations, cracks or grain boundary structures poses fundamental challenges beyond the scope of the current analysis, necessitating also the development of new theory.

\subsection*{Outline} 
The paper is organized as follows: In Section~\ref{sec:pre}, we provide an overview of the variational formulation for the equilibration of crystalline defects and review the result on multipole expansions (cf.~\cite[Theorem 3.1]{2021-defectexpansion}) that provides a general characterization of the discrete elastic far-fields surrounding point defects. In Section~\ref{sec:moments}, we propose a numerical framework that leverage the multipole expansions to accelerate the simulation of crystalline defects. We utilize continuous multipole expansions instead of discrete ones to obtain an efficient implementation. In Section~\ref{sec:numexp}, we apply our main algorithm (cf. Algorithm~\ref{alg:moment_iter_a}) to various prototypical numerical examples of point defects. Section~\ref{sec:conclusion} presents a summary of our work and future directions. The proofs as well as additional analysis that can aid in understanding the main idea of this paper, are given in Section~\ref{sec:proofs}. 

\subsection*{Notation}

We use the symbol $\langle\cdot,\cdot\rangle$ to denote an abstract duality
pairing between a Banach space and its dual. The symbol $|\cdot|$ normally
denotes the Euclidean or Frobenius norm, while $\|\cdot\|$ denotes an operator
norm. We denote $A\backslash\{a\}$ by
$A\backslash a$, and $\{b-a~\vert ~b\in A\}$ by $A-a$. For $E \in C^2(X)$, the first and second variations are denoted by
$\<\delta E(u), v\>$ and $\<\delta^2 E(u) v, w\>$ for $u,v,w\in X$.

We write $|A| \lesssim B$ if there exists a constant $C$ such that $|A|\leq CB$, where $C$ may change from one line of an estimate to the next. When estimating rates of decay or convergence, $C$ will always remain independent of the system size, the configuration of the lattice and the the test functions. The dependence of $C$ will be normally clear from the context or stated explicitly. 

Given a $k$-tuple of vectors in $\R^d$, $\boldsymbol{\sigma}=(\sigma^{(1)}, \dots, \sigma^{(k)} ) \in (\R^d)^k$. We write the $k$-fold product $\boldsymbol{\sigma}^\otimes \in (\R^d)^{\otimes k}$ as
\[
\boldsymbol{\sigma}^\otimes := \bigotimes_{i=1}^k \sigma^{(i)} := \sigma^{(1)} \otimes \cdots \otimes \sigma^{(k)}.
\]
Similarly, we write $v^{\otimes k} := v \otimes ... \otimes v \in (\R^d)^{\otimes k}$ for $v \in \R^d$.

For any $\boldsymbol{\sigma}\in(\R^d)^k$, we define the symmetric tensor product by \[\boldsymbol{\sigma}^\odot:=  \sigma^{(1)} \odot \cdots \odot \sigma^{(k)} := {\rm sym~} \boldsymbol{\sigma}^\otimes := \frac{1}{k!} \sum_{g \in S_k} g(\boldsymbol{\sigma})^\otimes, \]
where $S_k$ is the usual symmetric group of all permutations acting on the integers $\{1,\ldots,k\}$ and $g(\boldsymbol{\sigma}) := (\sigma^{(g(1))}, \dots, \sigma^{(g(k))} )$ for any $g \in S_k$ and $\boldsymbol{\sigma} \in (\R^d)^k$. 

The natural scalar product on $(\R^d)^{\otimes k}$ is denoted by $A : B$ for $A,B \in (\R^d)^{\otimes k}$, which is defined to be the linear extension of
 \[\boldsymbol{\sigma}^\otimes : \boldsymbol{\rho}^\otimes := \prod_{i=1}^k \sigma^{(i)} \cdot \rho^{(i)} .\]

For two second-order tensors $\CC, \UU \in (\R^{A})^{\otimes k}$, given specifically as a sum $\CC = \sum_{\boldsymbol{\rho} \in A^{k}}  \CC_{\boldsymbol{\rho}} E_{\boldsymbol{\rho}}$ with $E_{\boldsymbol{\rho}}$ the natural basis of the space ${\R^d}^k$, we can then write 
\[
\CC \colon \UU = \sum_{\boldsymbol{\rho}\in A^k} \CC_{\boldsymbol{\rho}} \UU_{\boldsymbol{\rho}}.
\]

\section{Background: Equilibrium of crystalline defects and its multipole expansions}
\label{sec:pre}
Our work concerns the modeling of crystalline defects, with particular emphasis on single point defects embedded within a homogeneous crystalline bulk, a setting that allows a detailed and rigorous analysis of our approach. 
To motivate the formulation of our main results in this context, we first review and adapt the framework introduced in~\cite{chen19, EOS2016, 2014-dislift} in Section~\ref{sec:sub:model}. Subsequently, in Section~\ref{sec:sub:theory}, we provide a brief summary of the multipole expansion of equilibrium configurations proposed in~\cite{2021-defectexpansion}, which serves as the cornerstone for this work.

\subsection{Equilibrium of crystalline defects}
\label{sec:sub:model}

Let $d\in\{2,3\}$ be the dimension of the system. A homogeneous crystal reference configuration is given by the Bravais lattice $\La=\mathsf{A}\Z^d$, for some non-singular matrix $\mathsf{A} \in \mathbb{R}^{d \times d}$. We admit only single-species Bravais lattices. There are no conceptual obstacles to generalising our work to multi-lattices, however, the technical details become more involved. The reference configuration with defects is a set $\Ladef \subset \R^d$. The mismatch between $\Ladef$ and $\La$ represents possible defected configurations. We assume that the defect cores are localized, that is, there exists $R^{\rm def}>0$, such that $\Lambda^{\rm def} \backslash B_{R^{\rm def}} = \Lambda \backslash B_{R^{\rm def}}$. We refer to Figure \ref{figs:lattice} for a two dimensional example with $\mA$ defined by \cite[Eq.~(4.3)]{ortner2023framework} specifying a triangular lattice.

\begin{figure}[!htb]
\begin{center}
	\includegraphics[height=4.5cm]{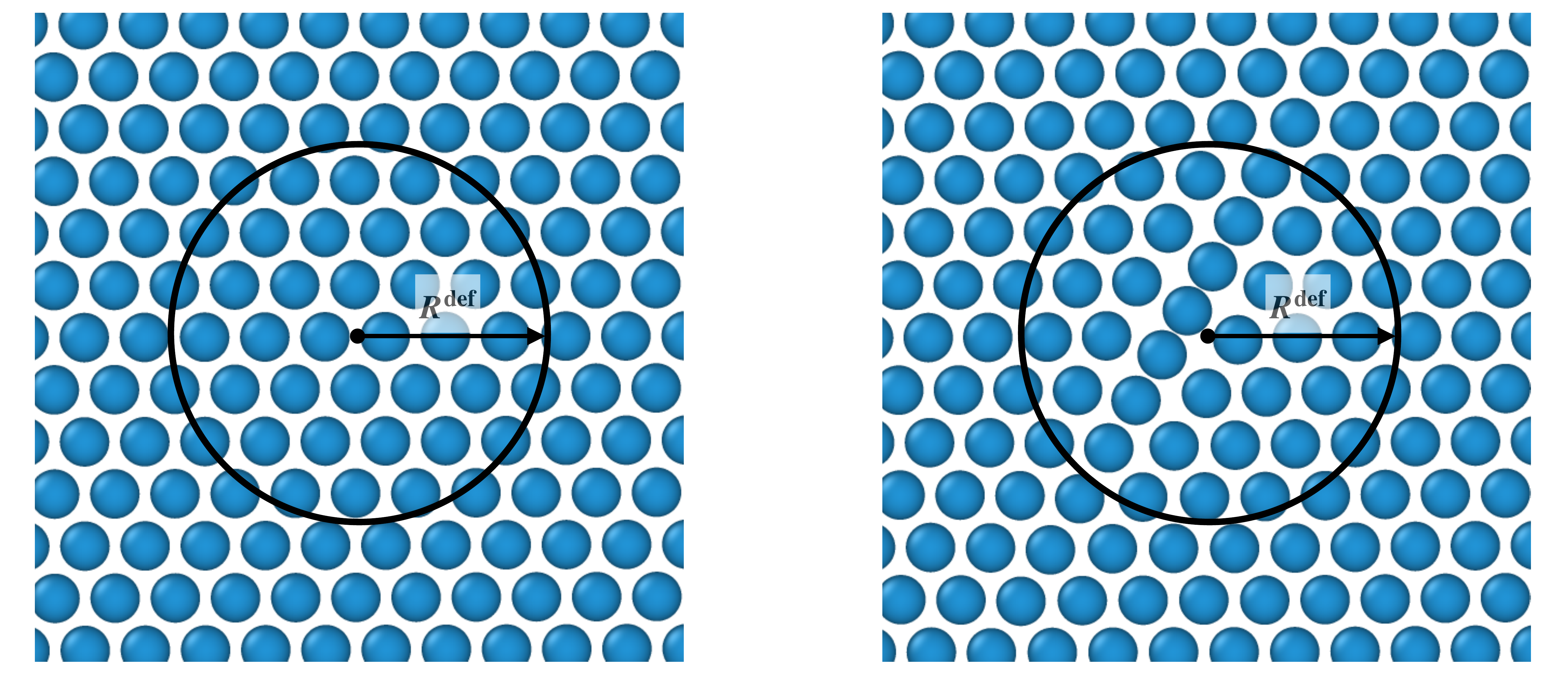}\qquad
	\caption{2D triangular lattice: Reference lattice $\Lambda$ (left); Defective lattice $\Ladef$ with one self-interstitial atom inside $B_{R^{\rm def}}$ (right).}
	\label{figs:lattice}
\end{center}
\end{figure}

The displacement of the infinite lattice $\Lambda$ is a map $u \colon \La \to \R^N$. Typically, $N = d$, but both $N < d$ and $N > d$ arise e.g. in dimension-reduced models. For $\ell, \rho \in \La$, we denote discrete differences by $D_\rho u (\ell) := u(\ell + \rho) - u(\ell)$. To introduce higher discrete differences we denote by $D_{\boldsymbol{\rho}} = D_{\rho_1} \cdots D_{\rho_j}$ for a $\boldsymbol{\rho} = (\rho_1, ..., \rho_j) \in \La^j$. For a subset $\Rc_{\ell} \subset \Lambda-\ell$, we
define $Du(\ell) := D_{\Rc_{\ell}} u(\ell) := \big(D_{\rho} u(\ell)\big)_{\rho\in\Rc}$. For the sake of simplicity, we assume throughout that $\Rc_{\ell}$ is {\em finite} for each $\ell\in\La$. An extension to infinite interaction range is possible but involves additional technical complexities~\cite{chen19, chen17}. 

We define two useful discrete energy space by 
\begin{align*}
  \HH(\Ladef) :=&~ 
  \big\{ 
    u : \Ladef \to \R^N ~\big|~ \|Du\|_{\ell^2} < \infty
  \big\}, \nonumber \\
  \HHc(\Ladef) :=&~ 
    \big\{ 
      u : \Ladef \to \R^N ~\big|~ {\rm supp}(Du) \text{~bounded~}
    \big\},
\end{align*}
where $\HHc$ is a dense subspace of $\HH$ with compact support. Analogously, $\HH(\La)$ and $\HHc(\La)$ can be defined for displacements on homogeneous lattice.

We consider the site potential to be a collection of mappings $V_{\ell}:(\R^N)^{\Rc_\ell}\rightarrow\R$, which represent the energy distributed to each atomic site. We make the following assumption on regularity and symmetry:
$V_\ell \in C^K((\R^N)^{\Rc_\ell})$ for some $K$ and $V_\ell$ is homogeneous outside the defect region $B_{R^{\rm def}}$, namely, $V_\ell = V$ and $\Rc_\ell = \Rc$ for $\ell \in \Lambda \setminus B_{R^{\rm def}}$. Furthermore, $V$ and $\Rc$ have the following point symmetry: $\Rc = -\Rc$, it spans the lattice $\textrm{span}_\Z \Rc= \La$, and $V(\{-A_{-\rho}\}_{\rho\in\Rc}) = V(A)$. We refer to~\cite[\S 2.3 and \S 4]{chen19} for a detailed discussion of those assumptions and symmetry.

The potential energy under the displacements field are given by
\begin{align}\label{eq:E}
   \mathcal{E}^{\rm def}(u) := \sum_{\ell \in \Ladef} \Big[V_\ell \big(D_{\mathcal{R}_\ell} u(\ell)\big) - V_\ell\big({\bf 0}\big)\Big], \qquad
   \mathcal{E}(u) := \sum_{\ell \in \La} \Big[ V \big(D u(\ell)\big) - V\big({\bf 0}\big) \Big].
\end{align}
It is shown in \cite[Lemma 1]{EOS2016} that $\mathcal{E}^{\rm def}$ (resp. $\mathcal{E}$) is well defined on $\HHc(\Ladef)$ (resp. $\HHc(\La)$) and has a unique continuous extension to $\HH(\Ladef)$ (resp. $\HH(\La)$). 

The main objective in this work is the characterisation of the far-field behaviour of lattice displacements $u : \La \to \R^N$ that are close to equilibrium. Following the works~\cite{2021-defectexpansion, EOS2016}, it is important to characterise the linearised residual forces $f(\ell) := H[u](\ell)$, where
\begin{equation}\label{eq:Hdefinition}
 H[u](\ell) := -\Div \big( \nabla^2 V({\bf 0})[Du]\big),
\end{equation}
with $\Div~A = -\sum_{\rho \in \Rc} D_{-\rho} A_{\cdot \rho}$ the discrete divergence for a matrix field $A \colon \La \to \R^{N\times \Rc}$. Given $i\in\N$, if $\ell \mapsto H[u](\ell) \otimes \ell^{\otimes i} \in \ell^1(\La)$, we define the $i$-th force moment 
\begin{equation}\label{eq:results:defn_Ij}
  \mathcal{I}_i[u] = \sum_{\ell \in \La} H [u](\ell) \otimes \ell^{\otimes i},
\end{equation}
which serves as a key concept in the following analysis and algorithms.

We assume throughout that the Hamiltonian $H=\delta^2 \mathcal{E}(\bf{0})$ is stable (cf.~\cite[Eq.~(4)]{2021-defectexpansion}). For stable operator $H$ there exists a {\it lattice Green's function} (inverse of $H$) $\mathcal{G}:\Lambda \rightarrow \R^{N\times N}$ such that
\begin{eqnarray}\label{eq:def_lattice_G}
    H[\mathcal{G}e_k](\ell) = e_k \delta_{\ell, 0}, \qquad \textrm{for}~1 \leq k \leq N. 
\end{eqnarray}
We write $\mathcal{G}_k:=\mathcal{G}e_k$ for simplicity.

The equilibrium displacement $\bar{u}^{\rm def} \in \HH(\Ladef)$ satisfies
\begin{equation}
  \label{eq:equil_point}
   \delta \mathcal{E}^{\rm def}(\bar{u}^{\rm def})[v] = 0 \qquad \forall v \in \HHc(\Ladef).
\end{equation}
For the purpose of the following analysis, it is advantageous to project $\bar{u}^{\rm def}$ onto the homogeneous lattice $\Lambda$ denoted by $\bar{u}$. The possible projections are not unique and will not influence the main results; we therefore refer to \cite[Section 3.1]{2021-defectexpansion} for details.

It was shown~\cite{chen19, EOS2016} that the equilibrium displacement has a generic decay $|D\bar{u}^{\rm def}(\ell)| \leq C |\ell|^{-d}$, which gives rise to a slow convergence of standard supercell methods with cell size~\cite{chen19, EOS2016}. The multipole expansion we introduce next gives additional information about the equilibrium far field and allows us to construct improved boundary conditions.

\subsection{Multipole expansion of equilibrium fields}
\label{sec:sub:theory}
In this section, we briefly review the results on the multipole expansion of the equilibrium $\bar{u}$ for point defects~\cite{2021-defectexpansion}, which provides the general structure for characterising the discrete elastic far-fields induced by defects.

Since $\bar{u} = \bar{u}^{\rm def}$ outside of the defect core, for $\lvert \ell \rvert$ large enough, we can obtain
\[
   \delta \mathcal{E}(\bar{u})(\ell) = \delta \mathcal{E}(\bar{u})[\delta_\ell] = 0,
\]
where $\delta_\ell(\ell') := \delta_{\ell\ell'}$. As a matter of fact, it is shown in \cite{EOS2016} that
\[
   \delta \mathcal{E}(\bar{u})[v] = (g, Dv)_{\ell^2} \qquad \forall v \in \HH(\La),
\]
where $g : \La \to \R^{d \times \Rc}$ with ${\rm supp}(g) \subset B_{\Rdef}$. 
%
The following Theorem taken from~\cite[Theorem 3.1]{2021-defectexpansion} is provided here for the sake of completeness.

\begin{theorem}\label{thm:pointdef}
   Choose $p \geq 0, J \geq 0$ and suppose that $V \in C^{K}(\R^{d \times  \Rc})$, such that 
   $K \geq J+2+ \max \{0,\lfloor \frac{p-1}{d} \rfloor\}$. Let $g : \La \to \R^{d \times \Rc}$ with compact support, and let
    $\bar{u} \in \HH(\La)$ such that
   \[
      \delta \mathcal{E}(\bar{u})[v] = ( g, D v )_{\ell^2}
      \quad \forall v \in \HHc(\La).
   \]
   Furthermore, let $\mathcal{S} \subset \Lambda$ be linearly independent with ${\rm span}_{\mathbb{Z}}\mathcal{S}=\Lambda$ and $\mathcal{G}:\La \rightarrow \R^{N\times N}$ be a lattice Green's function defined by \eqref{eq:def_lattice_G}. Then, there exist $u_i^{C} \in C^\infty$ and coefficients $b_{\rm exact}^{(i,k)}\in (\R^{\mathcal{S}})^{\odot i}$ such that 
\begin{eqnarray}\label{eq:exp_b}
\bar{u} = \sum_{i=d+1}^p u_i^C + \sum_{i=1}^{p} \sum_{k=1}^N b_{\rm exact}^{(i,k)} : D_{\rm \mathcal{S}}^i \mathcal{G}_k + r_{p+1},
\end{eqnarray}
where $u_i^C$ satisfies the PDEs in \cite[Eq.(59)]{2021-defectexpansion} for $d+1\leq i\leq p$ while $u_i^C=0$ for $0 \leq i \leq d$. Furthermore, for $j= 1, \dots,J$, the remainder term $r_{p+1}$ satisfies the estimate
   \begin{eqnarray}\label{eq:maindecay}
  \lvert D^j r_{p+1} \rvert \lesssim \lvert\ell\rvert^{1-d-j-p}\log^{p+1} (\lvert\ell\rvert).
   \end{eqnarray}
\end{theorem}

\begin{remark}\label{rem:pointdef}
Since $u_i^{C}=0$ for $0\leq i \leq d$, one can obtain the pure multipole expansion up to the order $p=d$
\begin{align} \label{eq:puremultipole}
\bar{u} = \sum_{i=1}^{d} \sum_{k=1}^N b_{\rm exact}^{(i,k)} : D_{\rm \mathcal{S}}^i \mathcal{G}_k + r_{d+1},
\ \text{where} \ 
\lvert D^j r_{d+1} \rvert \lesssim \lvert\ell\rvert^{1-2d-j} \log^{d+1} (\lvert\ell\rvert).
\end{align}
Moreover, as discussed in \cite[Remark 3.4]{2021-defectexpansion}, it is possible to reduce the number of logarithms in the estimate somewhat for all orders in \eqref{eq:maindecay} and \eqref{eq:puremultipole}. Despite this, we choose to include them as they do not have a significant impact on the core algorithms presented in this work.
\end{remark}

The foregoing theorem offers a framework for characterizing the discrete elastic far-fields encompassing crystalline point defects. This characterization relies on the decomposition~\eqref{eq:puremultipole}, with a particular focus on the multipole terms comprising the coefficients $b^{(i, k)}_{\rm exact}$ and the lattice Green's function $\mathcal{G}$. By determining these two crucial components, it becomes theoretically feasible to attain the desired regularity and decay of the remaining term $r_{d+1}$. This, in turn, facilitates the establishment of higher-order boundary conditions, which constitutes the primary objective of our present work.

As demonstrated in \cite[Lemma 5.6]{2021-defectexpansion}, the coefficients $b^{(i, k)}_{\rm exact}$ are theoretically obtainable through a linear transformation
\begin{eqnarray}\label{eq:bIrelation}
\big(\mathcal{I}_i(\bar{u})\big)_{\cdot k} = (-1)^i i! \sum_{\boldsymbol{\rho}\in \mathcal{S}^i} (b^{(i,k)}_{\rm exact})_{\boldsymbol{\rho}} \cdot \boldsymbol{\rho}^{\odot},
\end{eqnarray}
where the force moments are defined by \eqref{eq:results:defn_Ij}. For a detailed derivation of \eqref{eq:bIrelation} especially when $i=1,2,3$, we refer to Section~\ref{sec:sub:abIrelation} (cf.~\eqref{eq:bI}). It is worth noting, however, that both the force moments and the coefficients are defined globally, making them impractical to compute exactly. We will introduce an iterative procedure to approximate these quantities within a bounded domain, accompanied by a sharp estimate of the error committed by this truncation (cf.~Section~\ref{sec:sub:mominter}).

Regarding the lattice Green's function $\mathcal{G}$, it is important to note that the computational cost associated with $\mathcal{G}$, particularly its discrete higher-order derivatives, can be overwhelming. To establish an efficient numerical approach, we need to eliminate the necessity to work directly with $\mathcal{G}$. This is made feasible through a well-established connection between the continuous and lattice Green's functions, complete with rigorously controlled errors, as shown in \cite[Theorem 2.5 and Theorem 2.6]{2021-defectexpansion}. As a result, we can shift our focus away from handling the discrete coefficients $b^{(i,k)}_{\rm exact}$, and instead explore a pure continuous decomposition of $\bar{u}$ (cf.~\eqref{eq:exp_a}) using continuous Green's functions. Detailed construction will be provided in Section~\ref{sec:sub:coeffts}. 

\section{Moment Iterations and Continuous Computational Scheme}
\label{sec:moments}
A direct implication of Theorem~\ref{thm:pointdef} is the potential to accelerate standard cell approximations in solving the defect equilibration problem \eqref{eq:equil_point}, improving the relatively slow convergence with respect to the cell size in standard approaches~\cite{EOS2016}. In this section, we first revisit a conventional Galerkin approximation scheme for cell problems and point out its practical limitations. Following that, we introduce a numerical framework of integrating moment iterations and continuous Green's functions, which form the basis for implementing higher-order boundary conditions. 

\subsection{Accelerated convergence of cell problems}
\label{sec:sub:acc}
To apply the Galerkin approximation scheme, we define a family of restricted displacement spaces 
\begin{equation*}
   \mathcal{U}_R := \big\{  v : \Lambda \to \R^N ~\big|~
                    v(\ell) = 0 \text{ for $|\ell| > R$} \big\}, 
\end{equation*}
where atoms are clamped in their reference configurations outside a ball
with radius $R$. Then we can approximate \eqref{eq:equil_point} by the Galerkin projection as follows
\begin{equation}
   \label{eq:cellp:galerkin}
   \delta \mathcal{E}(\bar{u}_R)[v] = 0 \qquad \forall v \in \mathcal{U}_R,
\end{equation}
where $\bar{u}_R \in \mathcal{U}_R$. Under suitable stability conditions it is shown in \cite{EOS2016} that, for $R$ sufficiently large, a solution $\bar{u}_R$ exists that satisfies the explicit convergence rate  
\begin{equation} \label{eq:slow_convergence_cell}
   \| D\bar{u}_R - D\bar{u} \|_{\ell^2} \leq C R^{-d/2}.
\end{equation}
The rate is an immediate corollary of the decay estimate $|Dr_1(\ell)|\lesssim\lvert\ell\rvert^{-d}$ (let $p=0$ in \eqref{eq:maindecay} and ignore the logarithms). 

We introduce the following three steps to accelerate the cell problem \eqref{eq:cellp:galerkin}, motivated by the multiple expansions:
\begin{enumerate} 
  \item In general we replace the naive far-field {\it predictor}
   $\hat{u}_0=0$ with the higher-order continuum {\it predictor}
\[
  \hat{u}_p := \sum_{i = d+1}^p u_i^{\rm C},
\]
where $u_i^{\rm C}$ are given in Theorem~\ref{thm:pointdef}. (This step is skipped when $p \leq d$.)
   \item Then, the admissible {\it corrector} space is enlarged with the multipole moments 
  \begin{align}\label{eq:exact_space}
    \mathcal{U}_R^{(p)} := \bigg\{
        v : \Lambda \to \mathbb{R}^N \,\Big|\, & \,\,
        v = \hat{u}_p + \sum_{i = 1}^p \sum_{k = 1}^N b^{(i,k)} : D^i_{\mathcal{S}} \mathcal{G}_k
              + w, \nonumber \\[-1em] 
          & \text{for free coefficients $b^{(i,k)}$ and 
          $w \in \mathcal{U}_R$}
    \bigg\}.
  \end{align}
  That is, the {\it corrector} displacement is now parameterised by its values in 
  the computational domain $\Lambda \cap B_R$ and by the discrete coefficients $b^{(i,k)}$
  of the multipole terms.
  \item We consider the pure Galerkin approximation scheme: Find $\bar{u}_{p, R}\in\mathcal{U}^{(p)}_R$ such that 
\begin{equation} \label{eq:galerkin}
    \delta \mathcal{E}(\bar{u}_{p, R}) [v_R] = 0 \qquad \forall v_R \in   \{u - \hat{u}_p~|~u \in \mathcal{U}_R^{(p)}\}. 
\end{equation}
\end{enumerate}

Our numerical tests below are confined to the scenario where $p\leq d$ and thus $\hat{u}_p=0$. But for the sake of generality of our theorems we will still consider general $p$ which requires $\hat{u}_p$. This allows for broader applicability of our method in potential future work.

The following theorem provides the error estimates of the Galerkin approximation~\eqref{eq:galerkin} in both geometry and energy error, which integrates~\cite[Theorem 2]{EOS2016} with \cite[Theorem 3.7]{2021-defectexpansion}. The improved rate should be contrasted with the rate for a naive scheme \eqref{eq:slow_convergence_cell}. The proof is given in the Section~\ref{sec:sub:proofs3}. 

\begin{theorem}\label{th:galerkin}
  Suppose that $\bar{u}$ is a strongly stable solution of \eqref{eq:equil_point}; that is, there exists a stability constant $c_0 > 0$ such that 
  \[
      \delta^2 \mathcal{E}(\bar{u})[v,v] 
      \geq c_0 \| Dv \|^2_{\ell^2}, \qquad \forall v \in \HH(\Lambda),
  \]
  then, for $R$ sufficiently large, there exists a solution $\bar{u}_{p, R} \in \mathcal{U}_R^{(p)}$ to~\eqref{eq:galerkin} with $b^{(i,k)} = b^{(i,k)}_{\rm exact}$ and such that 
  \begin{align}
     \big\| D\bar{u} - D\bar{u}_{p,R} \big\|_{\ell^2}
     &\leq C_{\rm G} \cdot 
     R^{- d/2 - p} \cdot \log^{p+1}(R), \\
     \big|\mathcal{E}(\bar{u})-\mathcal{E}(\bar{u}_{p, R})\big| &\leq C_{\rm E} \cdot 
     R^{- d - 2p} \cdot \log^{2p+2}(R).
  \end{align}
\end{theorem}

The foregoing theorem is an important theoretical milestone, showcasing the accelerated convergence of cell problems that can in principle be achieved. However, the implementation of the scheme \eqref{eq:galerkin} is not feasible directly as the $b^{(i,k)}$ affect $v$ on the entire lattice and not just on a finite domain. The exact multipole tensors $b_{\rm exact}^{(i, \cdot)}$ can also not be computed trivially in practice (cf.~Section~\ref{sec:sub:theory}). In the next section, we will approximate $b^{(i,k)}_{\rm exact}$ within a bounded domain, accompanied by a sharp error estimate for this truncation.

\subsection{Multipole moment approximations}
\label{sec:sub:calc_moments}
We modify our computational scheme from a pure Galerkin approximation \eqref{eq:galerkin} to one that involves evaluating the multipole tensors within a finite domain. To achieve this, we impose a constraint by fixing the multipole tensors in the corrector space in advance leading to the approximate corrector space:
\begin{align}\label{eq:approx_space}
    \mathcal{U}_{b,R}^{(p)} := \bigg\{
        v : \Lambda \to \mathbb{R}^N \,\bigg|\, \,\,
        v = \hat{u}_p + \sum_{i = 1}^p \sum_{k = 1}^N b^{(i,k)} : D^i_{\mathcal{S}} \mathcal{G}_k
              + w, \quad 
          \text{for $b$ fixed and $w \in \mathcal{U}_R$ }
    \bigg\}.
  \end{align}
  Compared to \eqref{eq:exact_space}, the only difference is that we fix the moment tensor $b = (b^{(i,k)})_{i, k}$. This makes it possible to implement the scheme as now all degrees of freedom and computations are in a finite domain. The choice for $b^{(i,k)}$ should of course be an approximation $b^{(i,k)}_{\rm exact}$.

  We consider the corresponding Galerkin approximation scheme: Find $\bar{u}_{b, R} \in \mathcal{U}^{(p)}_{b, R}$ such that
\begin{equation} \label{eq:galerkin_approx}
    \delta \mathcal{E}(\bar{u}_{b, R}) [v_{R}] = 0 \qquad \forall v_{R} \in \mathcal{U}_{R}. 
\end{equation}

For future reference, let $b$ and $b_{\rm exact}$ be the collections of tensors $b^{(i,k)}$ and $b^{(i,k)}_{\rm exact}$ for all possible pairs $(i,k)$.  By considering a solution within the space $\mathcal{U}_{b,R}^{(p)}$ for a given $b$, we obtain a specific variation of Theorem~\ref{th:galerkin}. The proof will be given in the Section~\ref{sec:sub:proofs3}. 

\begin{lemma}\label{th:galerkinfixedb}
  Suppose that $\bar{u}$ is a strongly stable solution of \eqref{eq:equil_point}. Then, for $R$ sufficiently large, there exists a solution $\bar{u}_{b,R} \in \mathcal{U}_{b,R}^{(p)}$ to \eqref{eq:galerkin_approx} such that  
  \begin{eqnarray}\label{eq:ubR}
     \big\| D\bar{u} - D\bar{u}_{b,R} \big\|_{\ell^2}
     \lesssim R^{- d/2 - p} \cdot \log^{p+1}(R) + \sum_{i=1}^p \lvert b^{(i,\cdot)} - b_{\rm exact}^{(i,\cdot)} \rvert \cdot R^{1-d/2-i}.
  \end{eqnarray}
\end{lemma}

Compared to the error estimates shown in Theorem~\ref{th:galerkin}, it is worth noting that an additional term arises in \eqref{eq:ubR} as a consequence of introducing an approximation for the moment tensor. Leveraging the linear transformation~\eqref{eq:bIrelation}, our approach shifts the focus towards evaluating the force moments $\mathcal{I}_i$ instead of $b^{(i, \cdot)}$.

We first introduce a truncation operator following the constructions in~\cite{EOS2016, 2014-dislift}. Let $\eta_R: \Lambda\rightarrow\R$ be a smooth cut-off function such that $\eta_R=1$ for $\lvert \ell \rvert \leq R/3$, $\eta_R =0$ for $\lvert \ell \rvert > 2R/3$ and $\lvert \nabla^j \eta_R \rvert \leq C_j R^{-j}$ for $0\leq j\leq3$. Then, we define the truncated force moments \eqref{eq:results:defn_Ij} by 
\begin{equation} \label{eq:defn_IjR}
  \mathcal{I}_{i,R}[u] := \sum_{\ell \in \La} \big(H [u](\ell) \otimes \ell^{\otimes i}\big)\cdot \eta_R(\ell).
\end{equation}

We introduce the following theorem, offering a qualitative estimate of the moment error stemming from truncation. For the sake of simplicity of presentation, we provide a sketch of the proof here, reserving the technical details for Section~\ref{sec:sub:proofs3}.

\begin{theorem}\label{thm:moments}
For $1\leq i\leq p$, let $\mathcal{I}_{i}$ and $\mathcal{I}_{i, R}$ be the exact and approximate $i$-th moment defined by \eqref{eq:results:defn_Ij} and \eqref{eq:defn_IjR}, respectively. Under the conditions of Theorem~\ref{th:galerkinfixedb}, we have 
\begin{equation}\label{eq:momenterror}
    \big\lvert \mathcal{I}_{i,R}[\bar{u}_{b, R}] - \mathcal{I}_{i}[\bar{u}] \big\rvert \lesssim R^{i-1-p} \cdot \log^p(R) + R^{i-1} \cdot \lVert D\bar{u}_{b, R} - D\bar{u} \rVert_{\ell^2}^2 + \alpha_i(R) \cdot \lVert D\bar{u}_{b, R} - D\bar{u} \rVert_{\ell^2},
\end{equation}
where the factors $\alpha_i(R)$ for $i=1,\ldots,p$ are given by
\[
\alpha_i(R) := \begin{cases} 
          1 & i<1+d/2, \\
          \log(R) & i=1+d/2, \\
          R^{i-1-d/2} & i>1+d/2. 
       \end{cases}
\]
\end{theorem}
\begin{proof}[Sketch of the proof]
We split the target moment error into two parts, that is, 
\[
\big\lvert \mathcal{I}_{i,R}[\bar{u}_{b, R}] - \mathcal{I}_{i}[\bar{u}] \big\rvert \leq \big\lvert \mathcal{I}_{i,R}[\bar{u}] - \mathcal{I}_{i}[\bar{u}] \big\rvert + \big\lvert \mathcal{I}_{i,R}[\bar{u}_{b, R}] - \mathcal{I}_{i, R}[\bar{u}] \big\rvert,
\]
where the first part originates from the truncation of moments that can be directly bounded by $\big\lvert \mathcal{I}_{i,R}[\bar{u}] - \mathcal{I}_{i}[\bar{u}] \big\rvert \lesssim R^{i-p-1}\cdot \log^{p}(R)$. For the second part, it arises from the difference between $\bar{u}$ and $\bar{u}_b$ measured in the energy norm, which can be estimated by comparing their linearised residual forces. The detailed proof will be given in Section~\ref{sec:sub:proofs3}.
\end{proof}

Theorem~\ref{thm:moments} plays a pivotal role in this work, as it offers a qualitatively sharp estimate of the approximation of moment tensors arising from the finite domain truncation. It furnishes a rigorous analysis pertaining to the defect dipole tensor (with the specific case of $i=2$), often referred to as the elastic dipole tensor. The significance of this analysis is underscored by its enduring relevance within the realm of defect physics~\cite{dudarev2018elastic, nazarov2016first}. Notably, our result is naturally extensible to encompass higher-order multipole tensors, including tripole ($i=3$) and quadrupole ($i=4$) tensors.

More importantly, it establishes the dependence of moment errors on the corresponding geometry error, with a specific convergence rate in terms of computational cell size. This observation serves as a motivation for achieving accelerated convergence of cell problems in practical implementation. 

\subsection{The moment iteration}
\label{sec:sub:mominter}
In light of the estimate~\eqref{eq:momenterror}, our objective is to refine the accuracy of the approximation for $b_{\rm exact}$ by incorporating the evolving approximated solution $\bar{u}_{b}$. This results in an iterative refinement process, where we initiate with $b=0$ and systematically improve the accuracy of $\lvert b^{(i,\cdot)} - b_{\rm exact}^{(i,\cdot)} \rvert$ by continually updating the current-stage approximated solution $\bar{u}_b$. 

As the iterative process unfolds, the dominant term indicated in~\eqref{eq:momenterror} gradually evolves into $R^{i-1-p} \log^{p}(R)$. This, in turn, results in the primary term within~\eqref{eq:ubR} being $R^{- d/2 - p} \log^{p+1}(R)$, which achieves the desired convergence rate as stated in Theorem~\ref{thm:pointdef}. The subsequent algorithm illustrates this step-by-step refinement procedure.

\begin{algorithm}[H]
\caption{Moment iteration}
\label{alg:moment_iter}
\hskip-8.0cm {\bf Prescribe} $p, d, m=0, b^{(i, \cdot)}_0=0$ for all $1\leq i\leq p$. 

\begin{algorithmic}[1]
\REPEAT
    \STATE{ \textit{Evaluate}: Given $b_m^{(i, \cdot)}$, compute $\bar{u}_{b_m, R}$ and the resulting moments $\mathcal{I}_{i,R}[\bar{u}_{b_m, R}]$. } 
    \STATE{ \textit{Update}: Given $\mathcal{I}_{i,R}[\bar{u}_{b_{m, R}}]$, compute new coefficients $b_{m+1}^{(i,\cdot)}$ by \eqref{eq:bIrelation}. Apply Lemma~\ref{thm:moments} to estimate the accuracy of $b_{m+1}^{(i,\cdot)}$. Let $m=m+1$.}
\UNTIL{\begin{equation}\label{eq:stop}
    \lvert b_m^{(i,\cdot)} - b_{\rm exact}^{(i,\cdot)} \rvert = O\big(R^{i-1-p}\cdot \log^{p+1}(R)\big)\qquad \textrm{for all}~1\leq i\leq p.
\end{equation}} 
Let $M=m$. Save $\bar{u}_{b_M, R}$ and $b_M$ as a collection of $b_M^{(i, \cdot)}$ for all $i$.
\end{algorithmic}
\end{algorithm}

We note that the stopping condition is purely algebraic. The number of iterative steps $M$ in Algorithm~\ref{alg:moment_iter} can and should be determined {\it a priori} based on the estimates in Theorem \ref{th:galerkinfixedb} and Theorem \ref{thm:moments}. As a general upper bound, it is straightforward to see that the stopping criterion \eqref{eq:stop} can be fulfilled within a maximum of $p$ moment iterations. This implies that the number of iteration steps is bounded by $M \leq p$. 

Taking into account the estimates \eqref{eq:ubR} and \eqref{eq:momenterror}, the stopping criterion \eqref{eq:stop} ensures that
 \[
     \lVert D\bar{u} - D\bar{u}_{b_M, R} \rVert_{\ell^2}
     \leq C_{\rm G} \cdot
     R^{- d/2 - p} \cdot \log^{p+1}(R).
  \]
This is the optimal estimate of the geometry error for cell problems that can be obtained (cf. Theorem \ref{th:galerkinfixedb}). The corresponding convergence rate for energy error can be achieved analogously by applying the techniques used in proving Theorem \ref{th:galerkin}. We have the following Corollary, followed by a concrete example (Algorithm~\ref{alg:3d_example}) for demonstration.

\begin{corollary}\label{th:galerkinbM}
Under the conditions of Theorem \ref{th:galerkinfixedb}, if $b_{M}$ is a collection of $b^{(i,\cdot)}_{M}$ constructed by the Algorithm~\ref{alg:moment_iter} for all $1\leq i\leq p$, then for $R$ sufficiently large, there exists a corrector $\bar{u}_{b_M} \in \mathcal{U}_{b_M, R}^{(p)}$ the solution to \eqref{eq:galerkin} such that 
 \begin{align}
     \lVert D\bar{u} - D\bar{u}_{b_M, R} \rVert_{\ell^2}
     &\leq C_{\rm G} \cdot 
     R^{-d/2-p} \cdot \log^{p+1}(R), \qquad \text{and} \\
     \big|\mathcal{E}(\bar{u})-\mathcal{E}(\bar{u}_{b_M, R})\big| &\leq C_{\rm E} \cdot R^{-d-2p} \cdot \log^{2p+2}(R).
 \end{align}
\end{corollary}

We give a concrete example of Algorithm~\ref{alg:moment_iter} to demonstrate how to iteratively achieve the results shown in Corollary~\ref{th:galerkinbM}. We choose $d=3$ and $p=3$, in this case the predictor $\hat{u}_p=0$. We will extensively employ this particular case in the numerical experiments conducted in Section~\ref{sec:numexp}. The procedure can be summarized as follows:

\begin{algorithm}[H]
\caption{A 3D example for $p=3$}
\label{alg:3d_example}
\begin{algorithmic}[1]
\STATE Let $b_0^{(1,\cdot)}=b_0^{(2,\cdot)}=b_0^{(3,\cdot)}=0$. Compute the zeroth-order corrector $\bar{u}_{b_0, R} = \bar{u}_R$. Lemma \ref{th:galerkinfixedb} shows
    \[\lVert D\bar{u} - D\bar{u}_{b_0, R} \rVert_{\ell^2} \lesssim R^{-3/2}. \]
\vskip0.2cm
\STATE Evaluate $b^{(i,\cdot)}_1$ based on $\mathcal{I}_{i,R}[\bar{u}_{b_0, R}]$ for $i=1,2,3$ by \eqref{eq:bIrelation}. Theorem~\ref{thm:moments} gives 
    \begin{align*}
        \lvert b_1^{(1,\cdot)}-b_{\rm exact}^{(1,\cdot)} \rvert \lesssim R^{-3/2}, \quad
        \lvert b_1^{(2,\cdot)}-b_{\rm exact}^{(2,\cdot)} \rvert \lesssim R^{-3/2},
        \quad
        \lvert b_1^{(3,\cdot)}-b_{\rm exact}^{(3,\cdot)} \rvert \lesssim R^{-1} \cdot \log^4(R), 
    \end{align*}
    where the estimate for $b_1^{(3,\cdot)}$ has already obtained the desired accuracy.
    Compute the first-order corrector $\bar{u}_{b_1, R}$ and we have
    \[\lVert D\bar{u} - D\bar{u}_{b_1, R} \rVert_{\ell^2} \lesssim R^{-3}. \]
\vskip-0.8cm
\STATE Evaluate $b^{(i,\cdot)}_2$ based on $\mathcal{I}_{i,R}[\bar{u}_{b_1, R}]$ for $i=1,2,3$ by \eqref{eq:bIrelation}. Theorem~\ref{thm:moments} gives 
    \begin{align*}
        \lvert b_2^{(1,\cdot)}-b_{\rm exact}^{(1,\cdot)} \rvert \lesssim R^{-3} \cdot \log^4(R), \quad
        \lvert b_2^{(2,\cdot)}-b_{\rm exact}^{(2,\cdot)} \rvert \lesssim R^{-2} \cdot \log^4(R), \quad
        \lvert b_2^{(3,\cdot)}-b_{\rm exact}^{(3,\cdot)} \rvert \lesssim R^{-1} \cdot \log^4(R),
    \end{align*}
    where all of them have the desired accuracy.
    The stopping criterion \eqref{eq:stop} is therefore satisfied. Compute the second-order corrector $\bar{u}_{b_2, R}$. Lemma \ref{th:galerkinfixedb} gives
    \[\lVert D\bar{u} - D\bar{u}_{b_2, R} \rVert_{\ell^2} \lesssim R^{-9/2}\cdot\log^4(R). \]
Hence, the desired convergence rate of the geometry error for $d=3$, $p=3$ is achieved by applying two iterations ($M=2$). 
\end{algorithmic}
\end{algorithm}

\subsection{Continuous coefficients of multipole expansion}
\label{sec:sub:coeffts}

In the iterative moment scheme described previously, we have theoretically tackled the challenge of accelerating the convergence of cell problems. However, it is essential to note that in practical applications, the computational cost associated with obtaining higher-order correctors becomes prohibitively demanding. This computational burden primarily arises from the necessity to solve for discrete moment coefficients using lattice Green's functions (cf.~\eqref{eq:approx_space}).

As mentioned in Section~\ref{sec:sub:theory}, it is pragmatically advantageous to adopt a continuous reformulation of the multipole expansion for practical implementation. This strategic choice allows us to circumvent the complexities associated with discrete Green's functions and their discrete derivatives. Our primary focus in this section is thus directed toward the continuous Green's functions and their continuous derivatives.

We restrict the remainder of this section to the case $p = d = 3$, aligned with most application scenarios and consistent with our numerical experiments shown in Section~\ref{sec:numexp}. Similar procedures apply for other dimensionalities. We leverage the theoretical framework presented in~\cite{2021-defectexpansion}, which deduces higher-order continuum approximations of the discrete Green's function $\mathcal{G}$~\cite[Theorem 3]{2021-defectexpansion}. For $p = d = 3$, this framework only necessitates the continuum Green's function $G_0$ and the corrector $G_1$ to resolve the atomistic-continuum error. The precise definition and further computational details are provided in Appendix~\ref{sec:sub:gf}. Based on this, we can establish a purely continuous multipole expansion. Although the subsequent lemma has been previously introduced in~\cite[Theorem 4]{2021-defectexpansion}, we include it here for the sake of completeness.

\begin{lemma} \label{thm:structurewithmomentscont}
Under the conditions of Theorem~\ref{thm:pointdef}, there exist $a^{(i,n,k)} \in (\R^{d})^{\odot i}$ with $1\leq i,k \leq 3$ and $n=0,1$ such that 
\begin{align}\label{eq:exp_a}
\bar{u} = \sum_{k=1}^3 \Bigg( \sum^{3}_{i=1} a^{(i,0,k)}:\nabla^i (G_0)_{\cdot k} + a^{(1,1,k)}:\nabla (G_1)_{\cdot k}\Bigg) + w 
\end{align}
 and for $j=1,2$, $\alpha\in\N_0$, the remainder decays as 
\begin{equation} \label{eq:structurewithmomentscontremainder}
 \lvert D^j w \rvert \lesssim {\lvert\ell\rvert}^{-4-j}  \log^{\alpha +1}(\lvert\ell\rvert).
\end{equation}
\end{lemma}

The lemma highlights the possibility of addressing the coefficients $a^{(i, n, k)}$ within continuous multipole expansions, as opposed to the conventional approach of solving for the coefficients $b^{(i, k)}$ in discrete multipole expansions (cf. Theorem~\ref{thm:pointdef}). Taking into account \eqref{eq:exp_b} with \eqref{eq:exp_a}, we can establish a specific relationship between $a^{(i, n, k)}$ and $b^{(i, k)}$ by utilizing the Taylor expansion of the discrete difference stencil. A detailed derivation of this relationship is provided in Section~\ref{sec:sub:abIrelation}. For $1\leq i,k \leq 3$ and $n=0,1$, the relation can be summarized as follows:
\begin{align}\label{eq:a-b}
    (a^{(1,0,k)})_{\cdot j} &= (a^{(1,1,k)})_{\cdot j} = \sum_{\rho \in \mathcal{R}} (b^{(1,k)})_\rho \cdot \rho_j, \nonumber \\[1ex]
    (a^{(2,0,k)})_{\cdot jm} &= \sum_{\rho, \sigma \in \mathcal{R}} (b^{(2,k)})_{\rho\sigma} \cdot \rho_j \sigma_m + \frac{1}{2} \sum_{\rho\in\mathcal{R}} (b^{(1,k)})_\rho \cdot \rho_j\rho_m \\[1ex]
    (a^{(3,0,k)})_{\cdot jmn} &= \sum_{\rho,\sigma,\tau \in \mathcal{R}} (b^{(3,k)})_{\rho\sigma\tau} \cdot \rho_j \sigma_m \tau_n + \frac{1}{2} \sum_{\rho,\sigma \in \mathcal{R}} (b^{(2,k)})_{\rho\sigma} \cdot (\rho_j\sigma_m\sigma_n + \rho_j\rho_m\sigma_n) \nonumber \\ 
    &\qquad \qquad \qquad + \frac{1}{6} \sum_{\rho\in \mathcal{R}} (b^{(1,k)})_\rho \cdot \rho_j\rho_m\rho_n. \nonumber
\end{align}
As an immediate consequence of the foregoing developments and in conjunction with~\eqref{eq:bIrelation}, we can establish a direct relationship between the continuous coefficients and moments. To be more precise, for each $i$ and $n$, if we denote $a^{(i,n,\cdot)}$ as a collection of $(a^{(i,n,k)})_{\cdot j}$ for all $j,k$, we can obtain
\begin{align}\label{eq:a-I}
    a^{(1,0,\cdot)} &= -\mathcal{I}_1[\bar{u}], \quad a^{(1,1,\cdot)} = -\mathcal{I}_1[\bar{u}], \nonumber \\[1ex]
    a^{(2,0,\cdot)} &= \frac{1}{2}\mathcal{I}_2[\bar{u}], \quad 
    a^{(3,0,\cdot)} = -\frac{1}{6}\mathcal{I}_3[\bar{u}].
\end{align}
This approach offers a practical method for computing the continuous coefficients $a^{(i,n,\cdot)}$ by means of force moments. The detailed derivations for this computational process are available in Section~\ref{sec:sub:abIrelation}. It is crucial to emphasize that the moment iteration introduced in the preceding section remains applicable to the continuous coefficients, $a^{(i,n,\cdot)}$, owing to their linear relationship~\eqref{eq:a-b}.
Consequently, the computation of the continuous coefficients $a^{(i,n,\cdot)}$ can be readily achieved using~\eqref{eq:a-I}. As a result, we propose a continuous version of Algorithm~\ref{alg:moment_iter} in the following. This adapted algorithm will serve as a fundamental component in our forthcoming numerical experiments (cf.~Section~\ref{sec:numexp}).

\begin{algorithm}[H]
\caption{Computation of {\it correctors} with higher-order boundary conditions}
\label{alg:moment_iter_a}
\begin{algorithmic}[1]
\STATE Compute the zeroth-order {\it corrector} : $\bar{u}_{0, R}=\bar{u}_R$ such that \eqref{eq:cellp:galerkin} holds. 
The convergence $\lVert D\bar{u} - D\bar{u}_{0, R} \rVert_{\ell^2} \lesssim R^{-3/2}$ is then obtained.

\vskip0.2cm
\STATE Evaluate $a^{(1,0)}_1, a^{(1,1)}_1, a^{(2,0)}_1, a^{(3,0)}_1$ by applying \eqref{eq:a-I} with $\mathcal{I}_{j, R}[\bar{u}_{0,R}]$ (cf.~\eqref{eq:a-I-R}). Compute the first-order far-field {\it predictor} (boundary condition) by Lemma~\ref{thm:structurewithmomentscont}
\begin{eqnarray}\label{eq:u1ff}
\hat{g}_1 := a^{(1,0)}_1:\nabla G_0 +a^{(1,1)}_1 : \nabla G_1 + a^{(2,0)}_1 :\nabla^2 G_0 + a^{(3,0)}_1:\nabla^3 G_0.
\end{eqnarray}
\vskip-0.8cm
\STATE Compute the first-order {\it corrector} : $\bar{u}_{1,R} \in \hat{g}_1 + \mathcal{U}_R$ such that 
\[
    \delta\mathcal{E}(\bar{u}_{1,R})[v_R] = 0 \qquad \forall v_R \in \mathcal{U}_R.
\]
The corresponding convergence is $\lVert D\bar{u} - D\bar{u}_{1, R} \rVert_{\ell^2} \lesssim R^{-3}$. 
\vskip0.2cm
\STATE Evaluate $a^{(1,0)}_2, a^{(1,1)}_2, a^{(2,0)}_2, a^{(3,0)}_2$ by applying \eqref{eq:a-I} with $\mathcal{I}_{j, R}[\bar{u}_{1,R}]$ (cf.~\eqref{eq:a-I-R}). Compute the second-order far-field {\it predictor} (boundary condition) by Lemma~\ref{thm:structurewithmomentscont}
\begin{eqnarray}\label{eq:u2ff}
\hat{g}_2 := a^{(1,0)}_2:\nabla G_0 +a^{(1,1)}_2 : \nabla G_1 + a^{(2,0)}_2 :\nabla^2 G_0 + a^{(3,0)}_2:\nabla^3 G_0.
\end{eqnarray}
\vskip-0.8cm
\STATE Compute the second-order {\it corrector} : $\bar{u}_{2,R} \in \hat{g}_2 + \mathcal{U}_R$ such that 
\[
    \delta\mathcal{E}(\bar{u}_{2,R})[v_R] = 0 \qquad \forall v_R \in \mathcal{U}_R.
\]
The desired accuracy $\lVert D\bar{u} - D\bar{u}_{2, R} \rVert_{\ell^2} \lesssim R^{-9/2}\cdot\log^4(R)$ is then achieved.
\end{algorithmic}
\end{algorithm}

To be more clear, the computation of continuous coefficients for $i=0,1$ in Algorithm~\ref{alg:moment_iter_a} follows:
\begin{align}\label{eq:a-I-R}
    a^{(1,0,\cdot)}_{i+1} &= -\mathcal{I}_{1,R}[\bar{u}_{i, R}], \quad a^{(1,1,\cdot)}_{i+1} = -\mathcal{I}_{1, R}[\bar{u}_{i, R}], \nonumber \\[1ex]
    a^{(2,0,\cdot)}_{i+1} &= \frac{1}{2}\mathcal{I}_{2, R}[\bar{u}_{i, R}], \quad 
    a^{(3,0,\cdot)}_{i+1} = -\frac{1}{6}\mathcal{I}_{3, R}[\bar{u}_{i, R}].
\end{align}

While the algorithm above only includes the convergence of geometry error, the energy error can be directly estimated by applying
\begin{eqnarray}\label{eq:GE}
|\mathcal{E}(\bar{u})-\mathcal{E}(\bar{u}_{i, R})\big| \lesssim \lVert D\bar{u} - D\bar{u}_{i, R} \rVert^2_{\ell^2} \qquad \textrm{for}~0\leq i\leq 2. 
\end{eqnarray}
In contrast to Algorithm~\ref{alg:moment_iter}, Algorithm~\ref{alg:moment_iter_a} not only provides a practical method for obtaining the correctors with optimal convergence rates, but it also clarifies the construction of higher-order far-field {\it predictors}, which is also referred as higher-order boundary conditions.

\section{Numerical Experiments}
\label{sec:numexp}

In this section, we apply the main algorithm (cf.~Algorithm~\ref{alg:moment_iter_a}) to a selection of representative numerical examples featuring point defects (in the topological sense). Our objective is to demonstrate improved convergence rates with computational cell size for the derived higher-order boundary conditions. These improved convergence rates can be attained without a substantial rise in computational complexity, thanks to the utilization of continuous Green's functions.

\subsection{Model problems}
\label{sec:sub:numer_model}

In all our numerical experiments, we employ tungsten (W), a material characterized by a body-centered cubic (BCC) crystal structure in its solid state. The interatomic interaction is modelled via an embedded atom model (EAM) potential~\cite{Daw1984a}. For these simulations, the cut-off radius is set at $r_{\rm cut} = 5.5$\AA, encompassing interactions up to the third neighbor in the crystal lattice.

We consider five characteristic instances of localized point defects, each with their core geometries depicted on the (001) plane, as illustrated in Figure~\ref{figs:geom}:

\begin{itemize}
	\item {\it Single vacancy (Figure~\ref{fig:sub:geom_singvac}):} a single vacancy located at the origin, defined by $\Lambda^{\rm def}:=\Lambda\setminus\{{\bm 0}\}$; 
	\item {\it Divacancy (Figure~\ref{fig:sub:geom_divac}):} two adjacent vacancies;
	\item {\it Interstitial (Figure~\ref{fig:sub:geom_int}):} an additional W atom located at the centre of a bond between two nearest neighbors, defined by $\Lambda^{\rm def}:=\Lambda \cup \{(r_0/4, r_0/4, r_0/4)\}$, where $r_0$ is the lattice constant of W;
	\item {\it Microcrack2 (Figure~\ref{fig:sub:geom_micro2}):} a row of five adjacent vacancies; 
	\item {\it Microcrack3 (Figure~\ref{fig:sub:geom_micro3}):} a row of seven adjacent vacancies. 
\end{itemize}

The two examples labelled ``micro-cracks'' are not cracks in the conventional sense. They serve as examples of localized defects with an anisotropic shape. We will see how this leads to a significant pre-asymptotic regime.

\begin{figure}[!htb]
    \centering
    \subfloat[Vacancy \label{fig:sub:geom_singvac}]{
    \includegraphics[height=4.91cm]{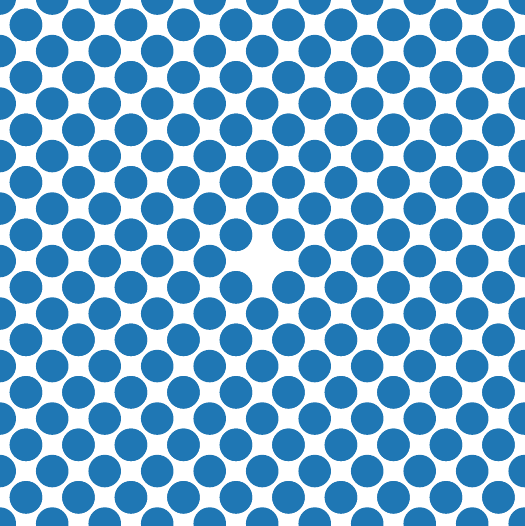}} \quad
    \subfloat[Divacancy \label{fig:sub:geom_divac}]{
    \includegraphics[height=4.91cm]{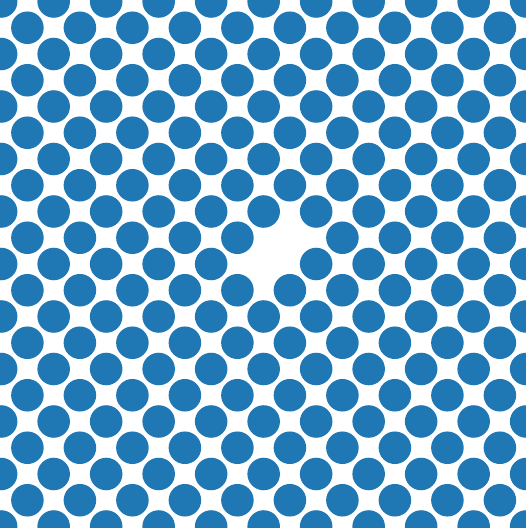}} \quad
    \subfloat[Interstitial \label{fig:sub:geom_int}]{
    \includegraphics[height=4.91cm]{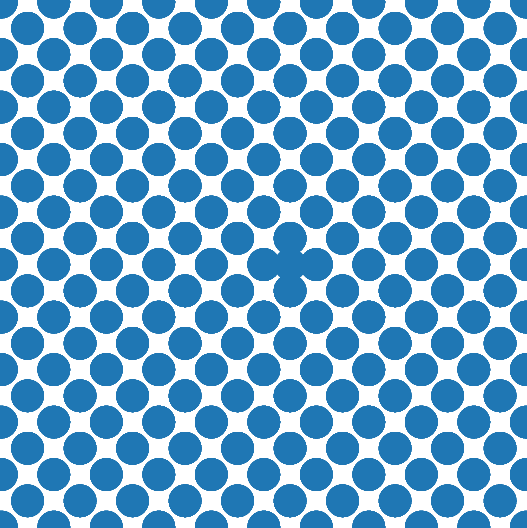}} \\
    \subfloat[Microcrack2 \label{fig:sub:geom_micro2}]{
    \includegraphics[height=4.91cm]{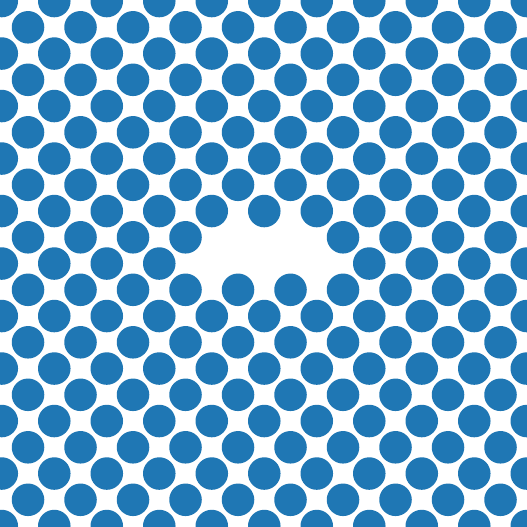}} \quad
    \subfloat[Microcrack3 \label{fig:sub:geom_micro3}]{
    \includegraphics[height=4.91cm]{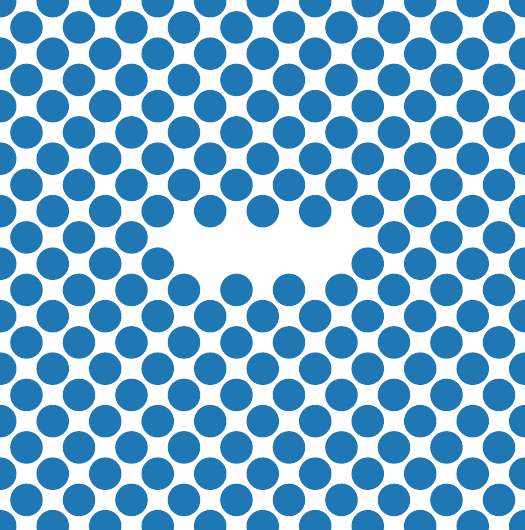}}
    \caption{Defect cores for the five cases considered in this work, illustrated on the (001) plane, serving as benchmark problems for the numerical tests.} 
    \label{figs:geom}
\end{figure}

For the computation of the equilibrium displacement in \eqref{eq:equil_point} with a high degree of accuracy, we utilize a preconditioned Limited-memory Broyden-Fletcher-Goldfarb-Shanno (LBFGS) algorithm, as outlined in \cite{liu1989limited}. Subsequently, we employ post-processing with a standard Newton scheme. The minimization process is halted when the force residual, measured in the $\ell^{\infty}$-norm, reaches a value of $10^{-8}$, i.e., $\|\nabla \mathcal{E}(u)\|_{\infty} < {\rm tol} = 10^{-8}$.

In practical implementation, the computation of the continuous Green's function $G_0$ and its first-order correction $G_1$ necessitates the application of Barnett's formula~\cite{barnett1972precise}. A comprehensive derivation can be found in Section~\ref{sec:sub:gf}. To obtain the higher-order derivatives of $G_0$ up to the third order, denoted as $\nabla^j G_0$ with $j=1,2,3$, we employ forward-mode automatic differentiation based on dual numbers~\cite{revels2016forward}.

\subsection{Convergence of cell problems}
\label{sec:sub:conv}

We perform the convergence studies for both geometry error $\|D\bar{u}-D\bar{u}_{i, R}\|_{\ell^2}$ and energy error $\big|\mathcal{E}(\bar{u})-\mathcal{E}(\bar{u}_{i, R})\big|$ by increasing the radius of computational domain $R$. The approximate equilibrium $\bar{u}_{i, R}$ are obtained iteratively using the moment iteration approach (cf.~Algorithm~\ref{alg:moment_iter_a}) with the zeroth-order ($i=0$), the first-order ($i=2$) and the second-order ($i=1$) boundary conditions. The reference solution $\bar{u}$ is computed by solving the numerical problem on a significantly larger computational domain with a radius of $R_{\rm dom}=100a_0$, where $a_0$ represents the lattice constant of material W. 

\subsubsection*{Decay of strains}

In this section, we verify the decay of the {\it correctors} in strains, which is a direct application of Theorem~\ref{thm:pointdef}. Figure~\ref{figs:decay} illustrates the decay of the strains $|D\bar{u}_{i, R_{\rm dom}}(\ell)|$ for different orders of {\it predictors} obtained using Algorithm~\ref{alg:moment_iter_a} for all the crystalline defects investigated in this study. The transparent dots represent the corresponding data points ($|\ell|, |D\bar{u}_{i, R_{\rm dom}}(\ell)|$) for $i=0,1,2$, while the solid curves depict their envelopes. The observed numerical improvement in decay for higher-order {\it predictors} aligns with our theoretical findings.

\begin{figure}[!htb]
    \centering
    \subfloat[Vacancy \label{fig:sub:decay_singvac}]{
    \includegraphics[height=4.5cm]{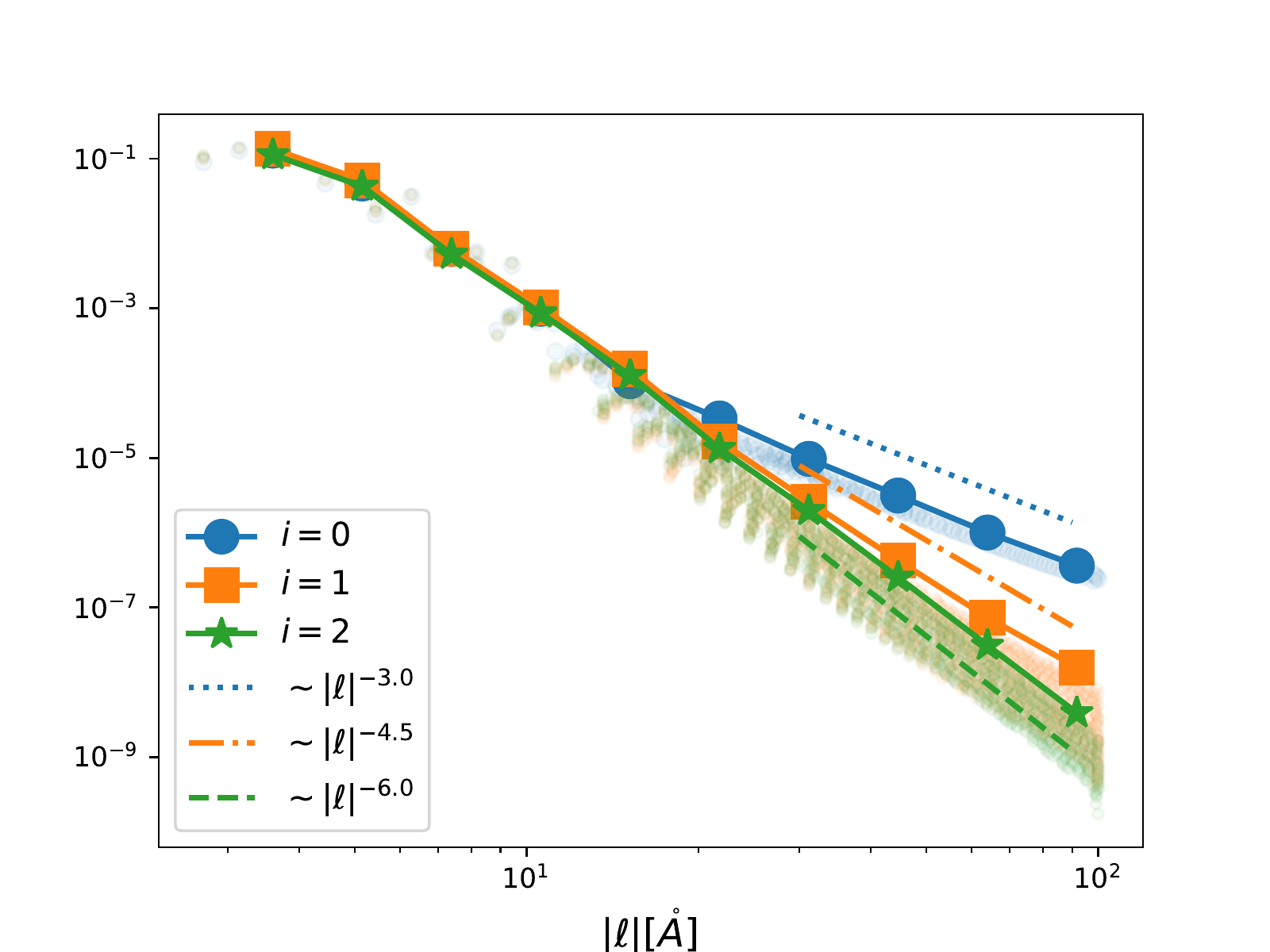}}~ 
    \subfloat[Divacancy \label{fig:sub:decay_divac}]{
    \includegraphics[height=4.5cm]{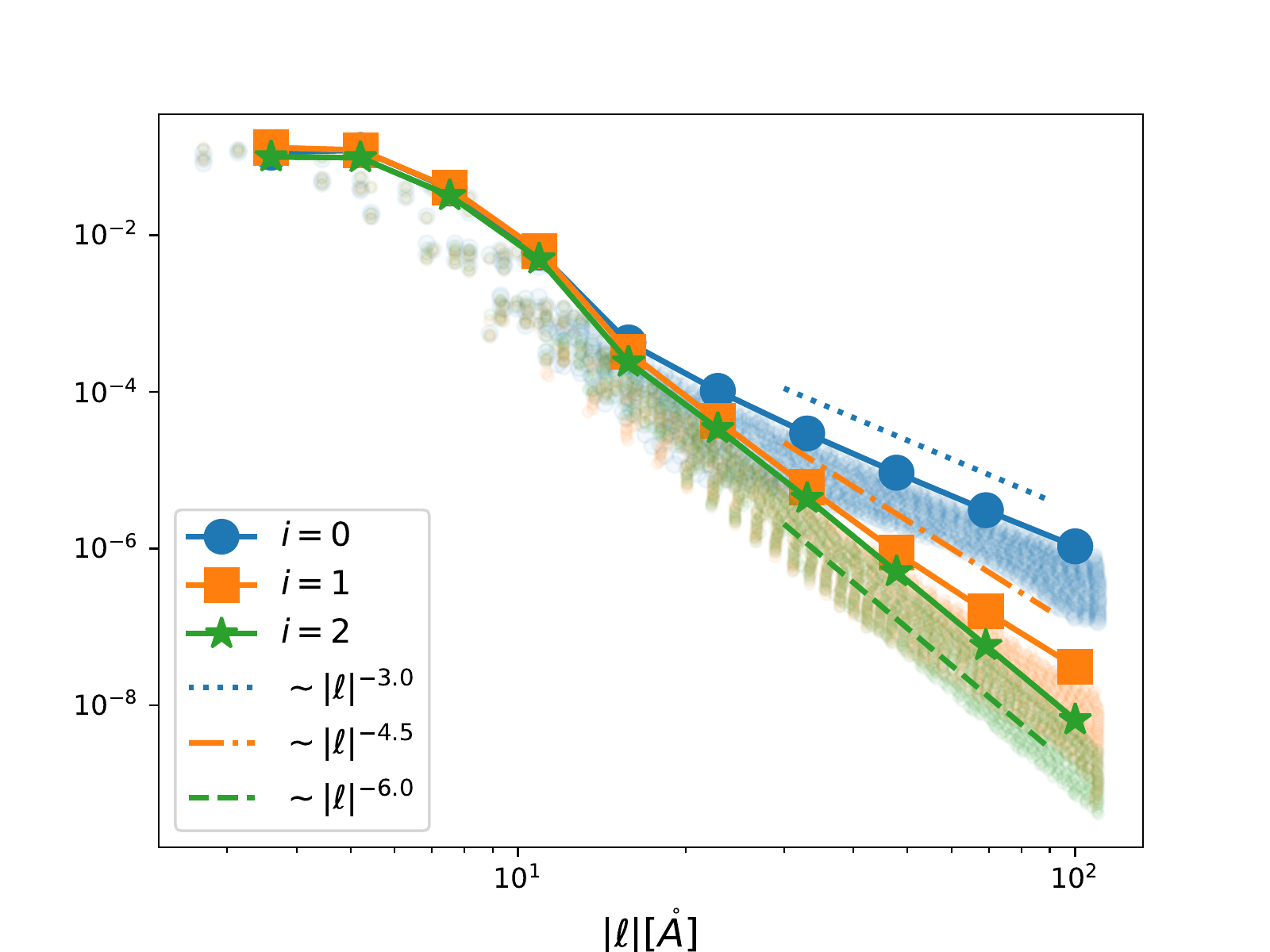}}~
    \subfloat[Interstitial \label{fig:sub:decay_int}]{
    \includegraphics[height=4.5cm]{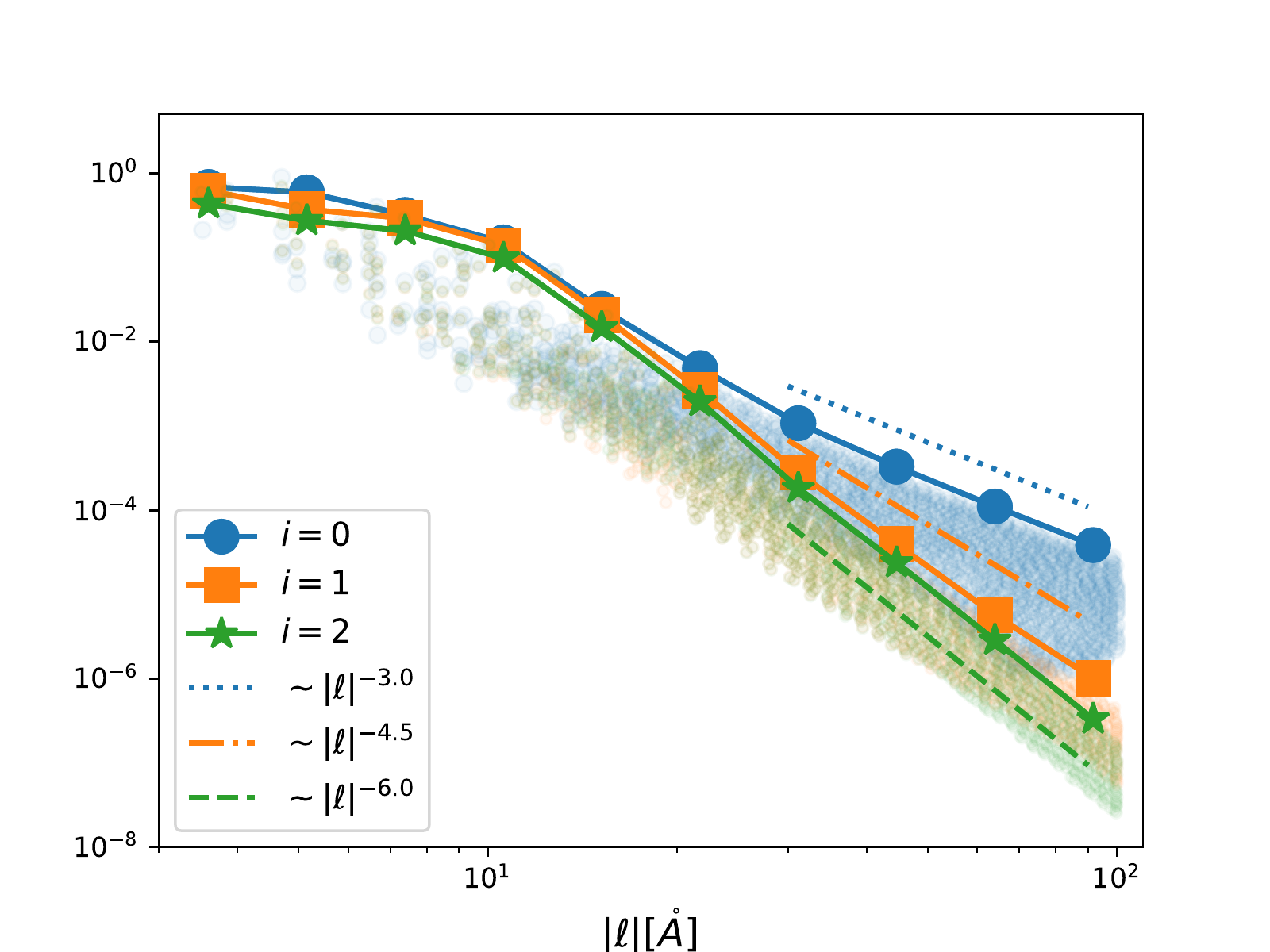}}\\
    \subfloat[Microcrack2 \label{fig:sub:decay_micro2}]{
    \includegraphics[height=4.5cm]{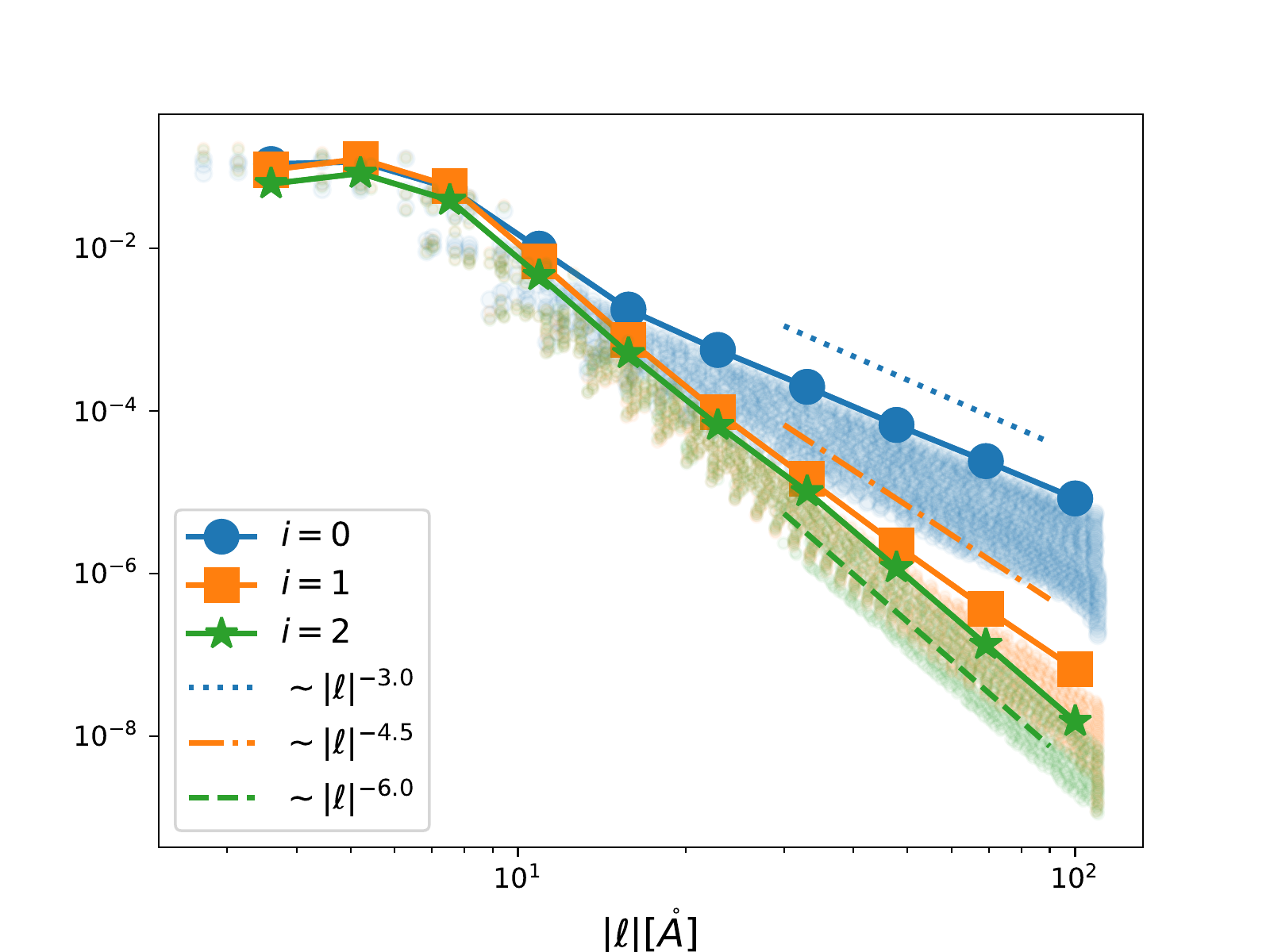}}~
    \subfloat[Microcrack3 \label{fig:sub:decay_micro3}]{
    \includegraphics[height=4.5cm]{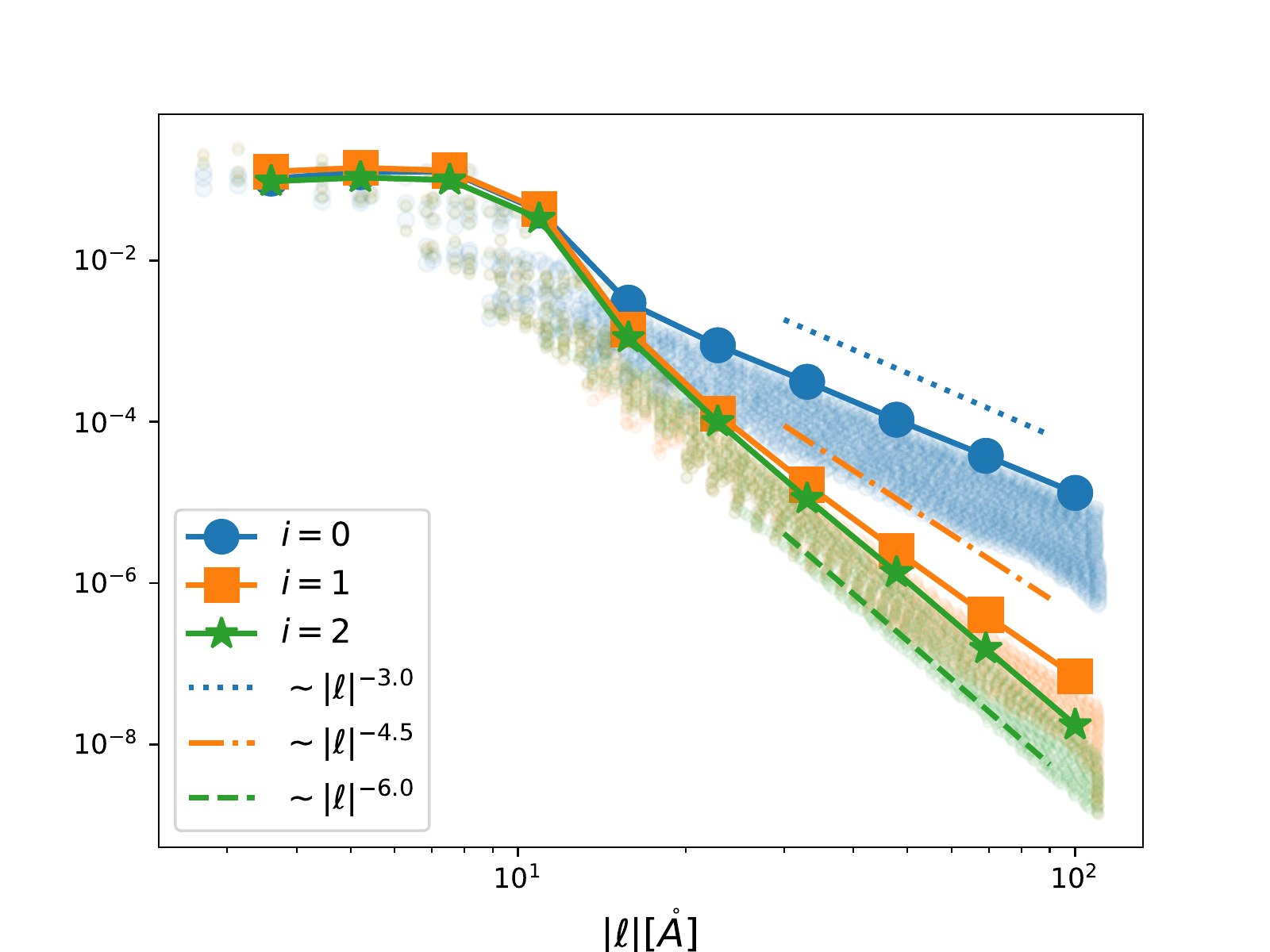}}~
    \caption{Decay of strains $\big|D\bar{u}_{i, R_{\rm dom}}(\ell)\big|$ for $i=0,1,2$ for all types of crystalline defect considered in this work.}
    \label{figs:decay}
\end{figure}

\subsubsection*{Moments convergence}

To begin, we investigate the convergence of force moments. The relative error of the $k$-th moment evaluated at $\bar{u}_{i, R}$ is defined by
\begin{eqnarray}\label{eq:num:moments_err}
{\rm ME}_{ki} := \frac{\big|\mathcal{I}_{k}[\bar{u}_{i, R}] - \mathcal{I}_k[\bar{u}]\big|}{\big|\mathcal{I}_k[\bar{u}]\big|},
\end{eqnarray}
where $\mathcal{I}_k$ is the $k$-th moment given by \eqref{eq:results:defn_Ij}. 

Figure~\ref{figs:conv_M} illustrates the convergence of the moments error \eqref{eq:num:moments_err} for $i=0,1$ and $k=1,2,3$, with respect to the domain size $R$, considering the vacancy, divacancy, and interstitial defects. The observed convergence behavior aligns with the theoretical predictions outlined in Theorem \ref{thm:moments}, demonstrating the potential for achieving accelerated convergence through the moment iteration process.
The numerical evidence not only validates the theoretical framework but also underscores the relevance of our results in defect physics~\cite{dudarev2018elastic, nazarov2016first}. Our approach employs an iterative method aimed at systematically enhancing the accuracy of approximate multipole tensor evaluations. The robust numerical results presented here for the defect dipole tensor remain of ongoing interest in point defect simulations. Importantly, our approach inherently accommodates the anisotropic case and facilitates extensions to higher-order multipole tensors.

\begin{figure}[!htb]
    \centering
    \subfloat[Vacancy \label{fig:sub:conv_M_singvac}]{
    \includegraphics[height=4.5cm]{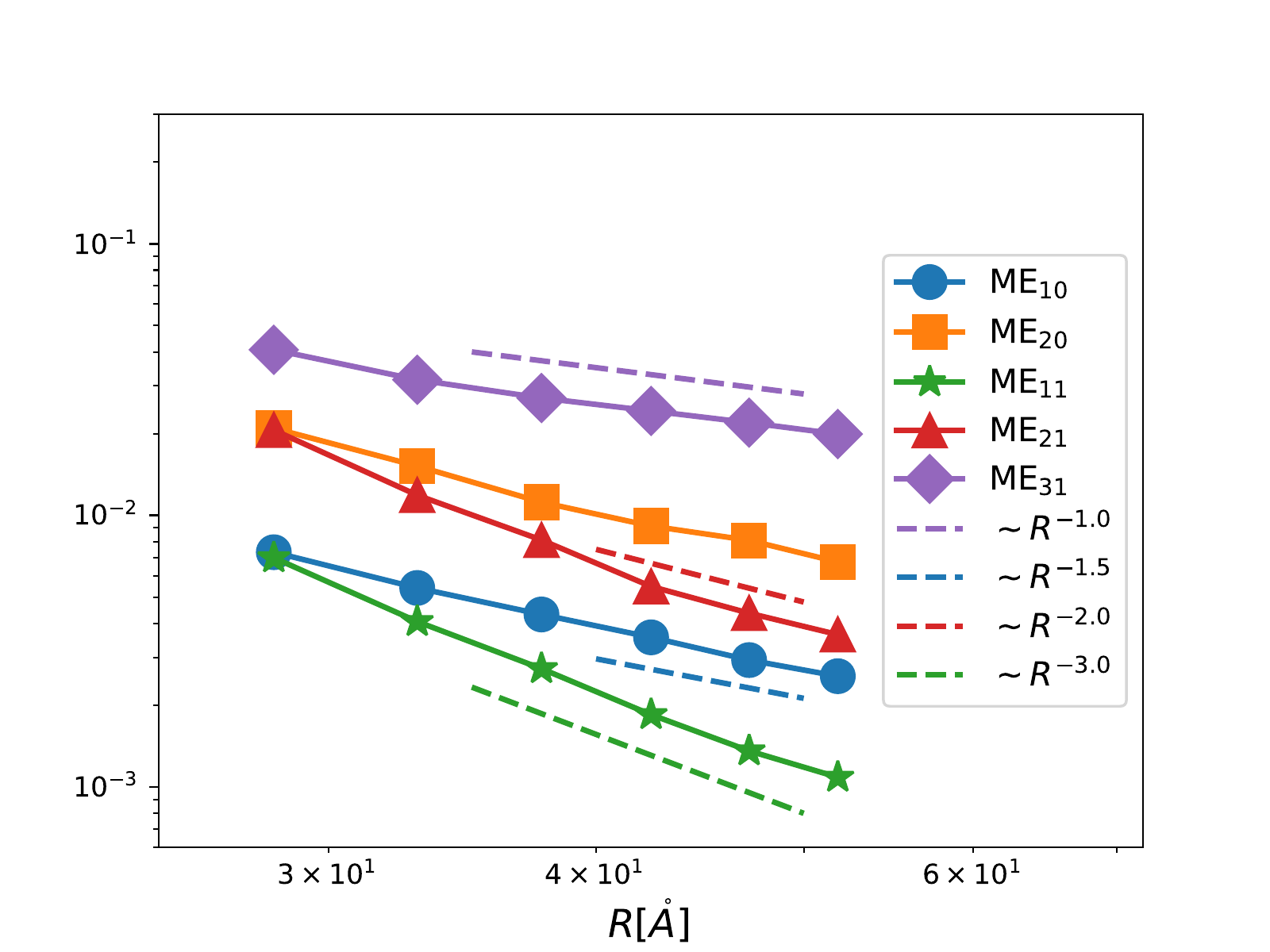}}
    \subfloat[Divacancy \label{fig:sub:conv_M_divac}]{
    \includegraphics[height=4.5cm]{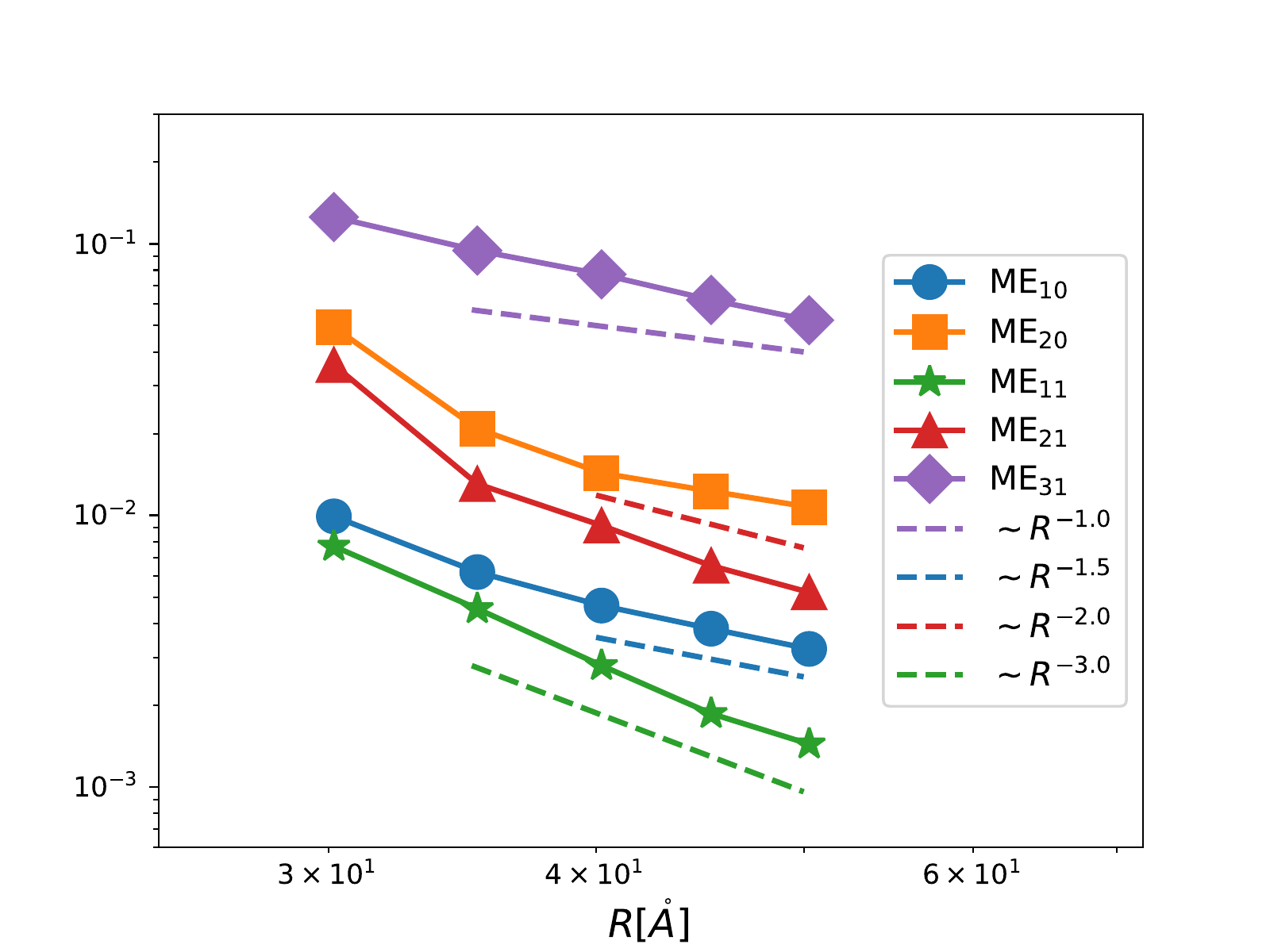}}
    \subfloat[Interstitial \label{fig:sub:conv_M_int}]{
    \includegraphics[height=4.5cm]{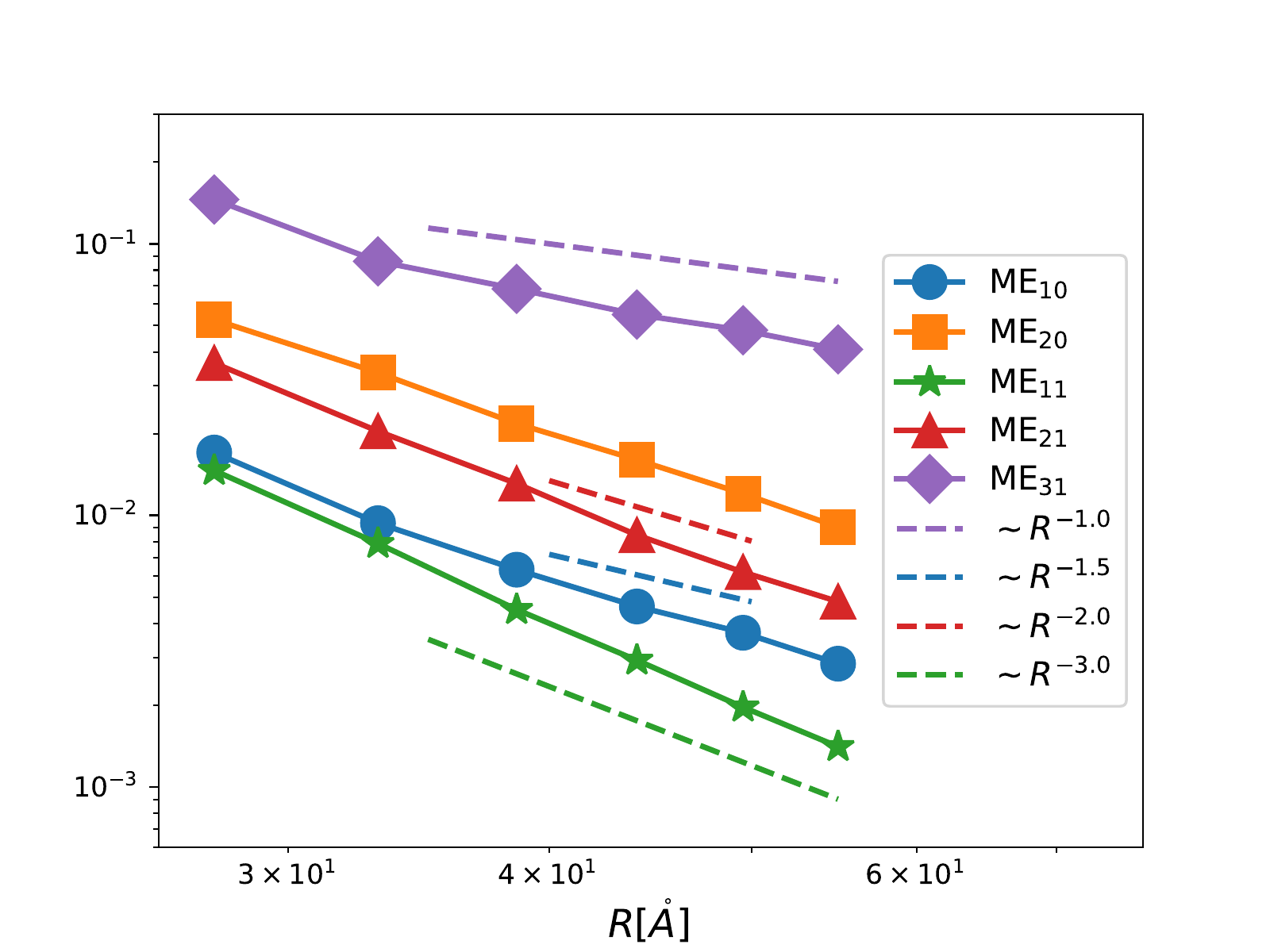}} \\
    \subfloat[Microcrack2 \label{fig:sub:conv_M_micro2}]{
    \includegraphics[height=4.5cm]{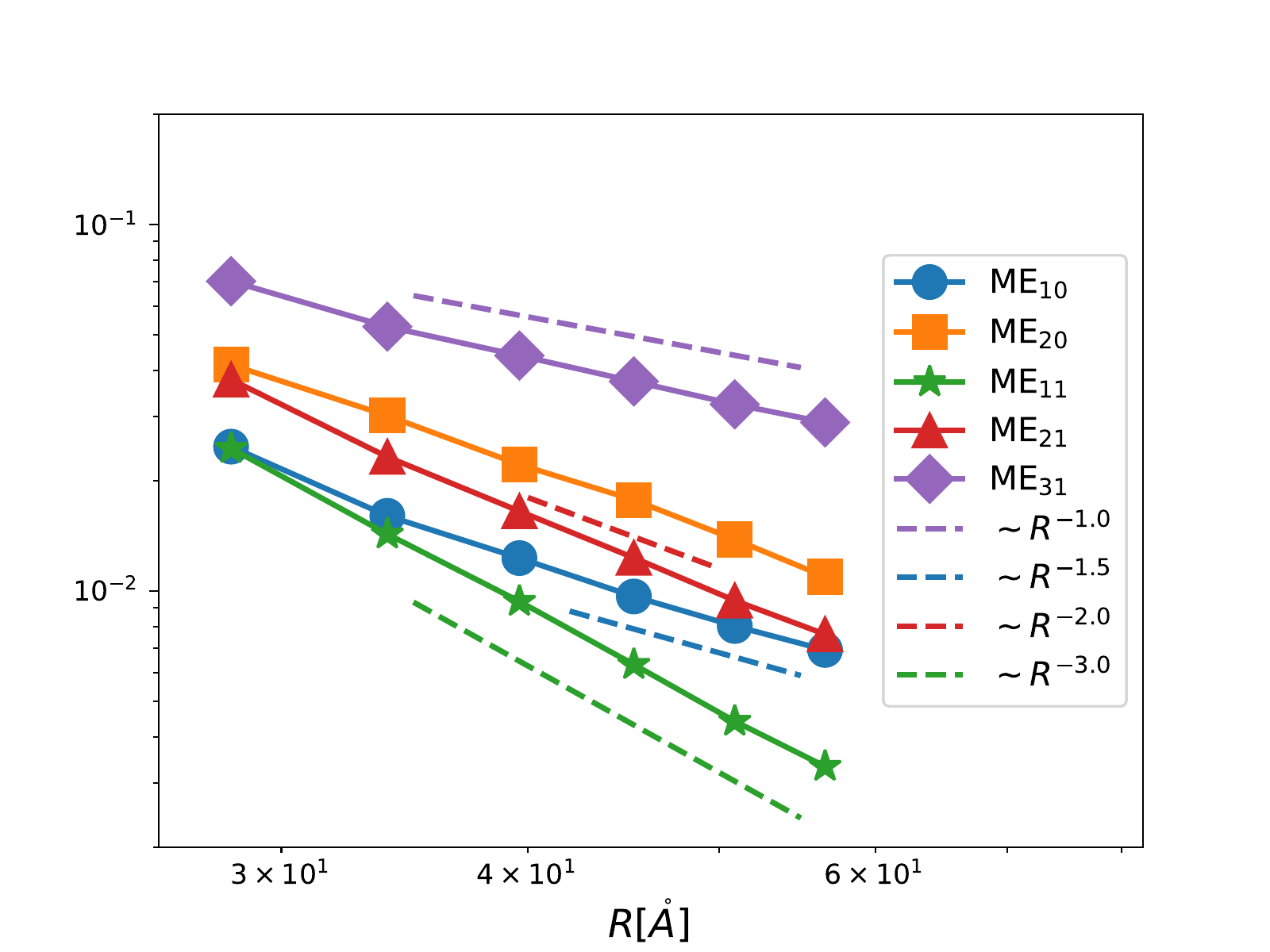}}
    \subfloat[Microcrack3 \label{fig:sub:conv_M_micro3}]{
    \includegraphics[height=4.5cm]{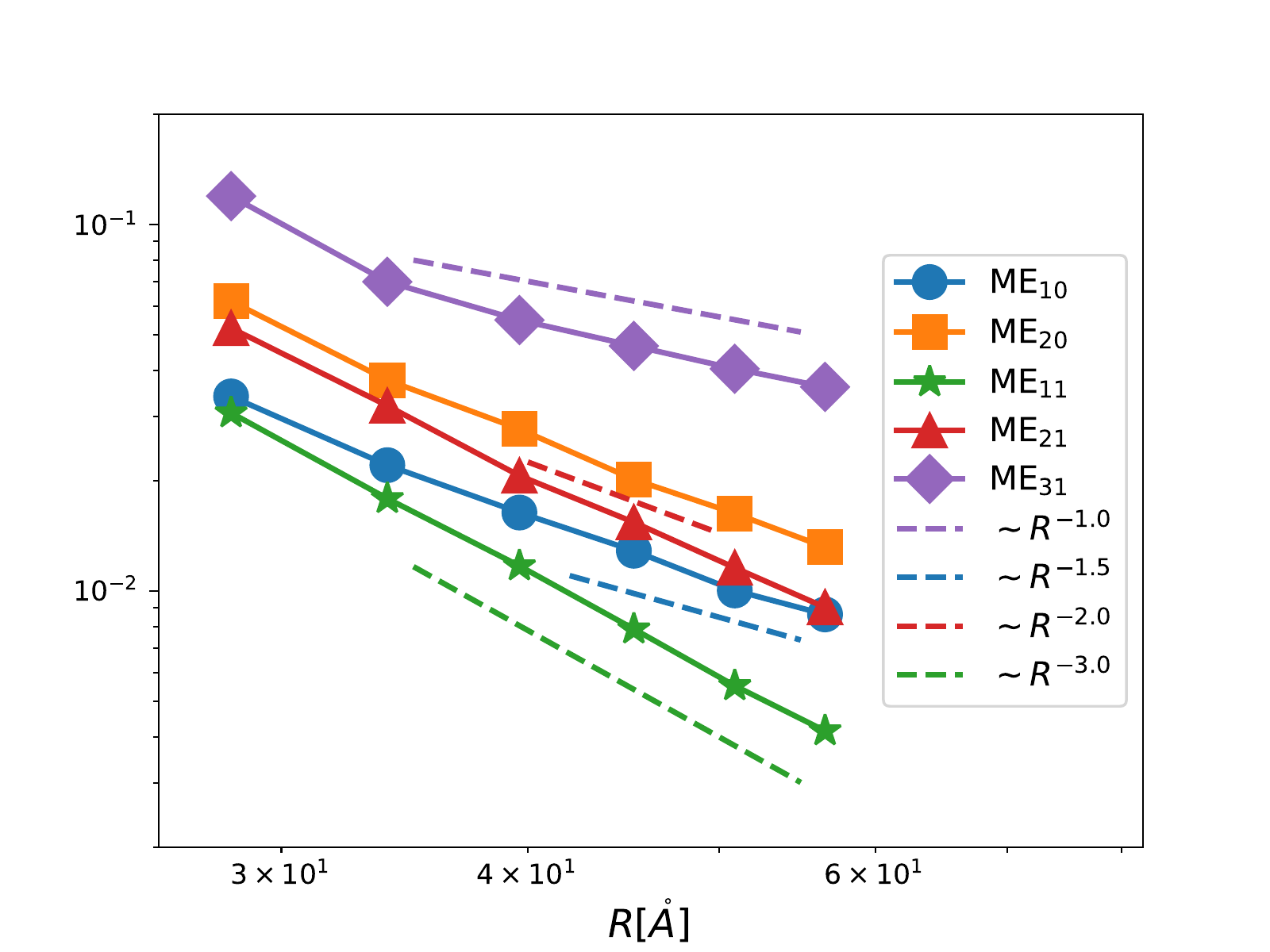}}
    \caption{Convergence of the relative moments error ${\rm ME}_{ki}$ defined by \eqref{eq:num:moments_err} against domain size $R$. The blue and green lines illustrate the accelerate convergence of the dipole moment tensor ($k=1$) when improved boundary conditions are considered.} 
    \label{figs:conv_M}
    \captionsetup{labelformat=empty}
\end{figure}

\subsubsection*{Geometry error}

The main observation that highlights the significance of this work is summarized in Figure~\ref{figs:conv_Du}, which showcases the convergence of the geometry error $\|D\bar{u} - D\bar{u}_{i, R}\|_{\ell^2}$ with respect to the domain size $R$. 

Figure~\ref{figs:conv_Du} provides a clear depiction of the improved convergence rates achieved through the implementation of higher-order boundary conditions, employing the moment iteration process. These observed convergence rates closely align with the theoretical predictions stated in Corollary~\ref{th:galerkinbM}. It is noteworthy that the higher-order convergence becomes more apparent as $R$ exceeds $30\mathring{\mathrm{A}}$. However, it is important to acknowledge that this poses increased computational challenges, especially in the context of electronic structure calculations. While addressing the associated computational costs falls outside the scope of this study, future research endeavors may explore potential solutions by combining the current scheme with the flexible boundary conditions presented in \cite{buze2021numerical}, which appear to reduce the range of the pre-asymptotic regime.

\begin{figure}[!htb]
    \centering
    \subfloat[Vacancy \label{fig:sub:conv_Du_singvac}]{
    \includegraphics[height=4.5cm]{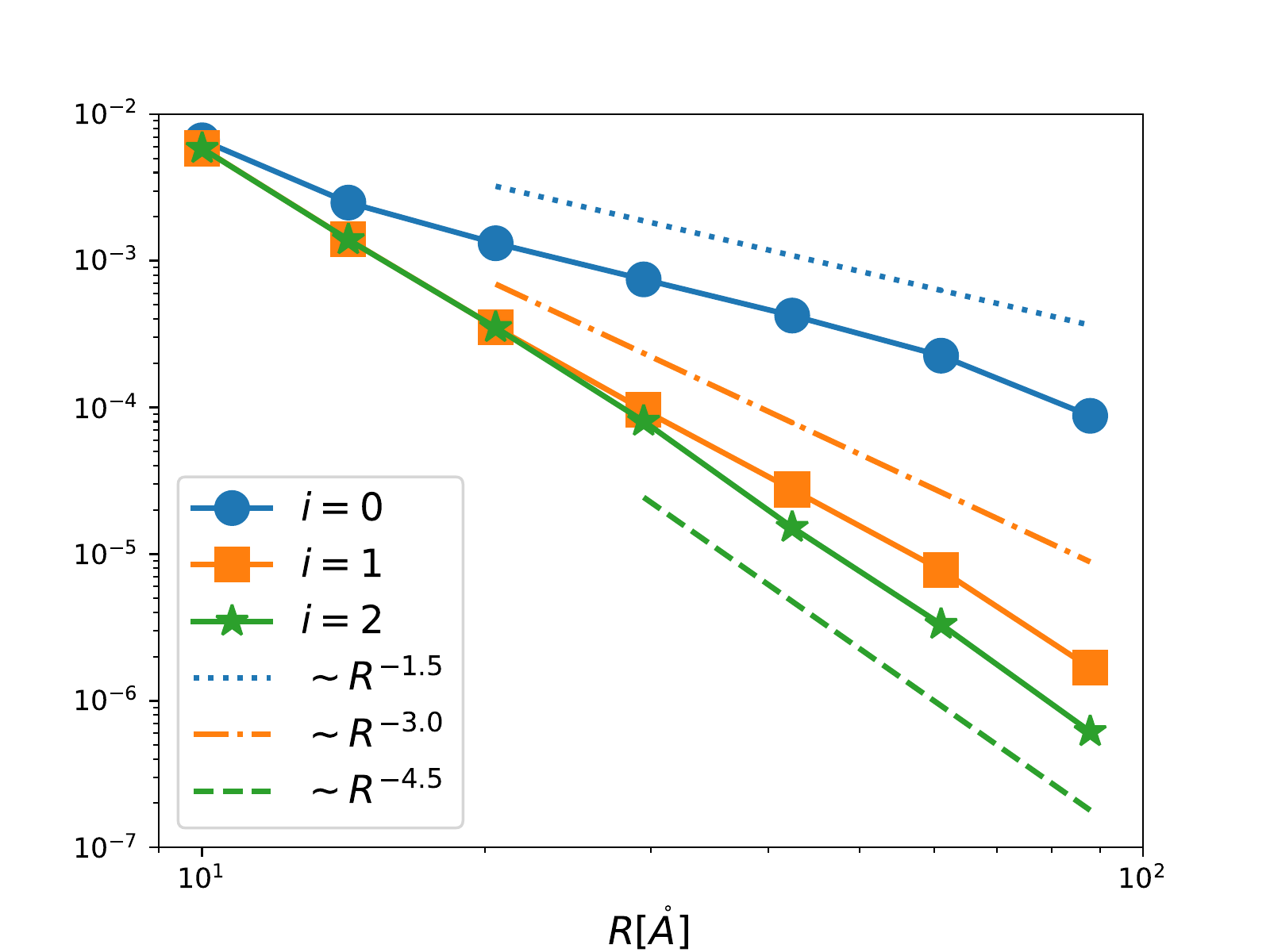}} 
    \subfloat[Divacancy \label{fig:sub:conv_Du_divac}]{
    \includegraphics[height=4.5cm]{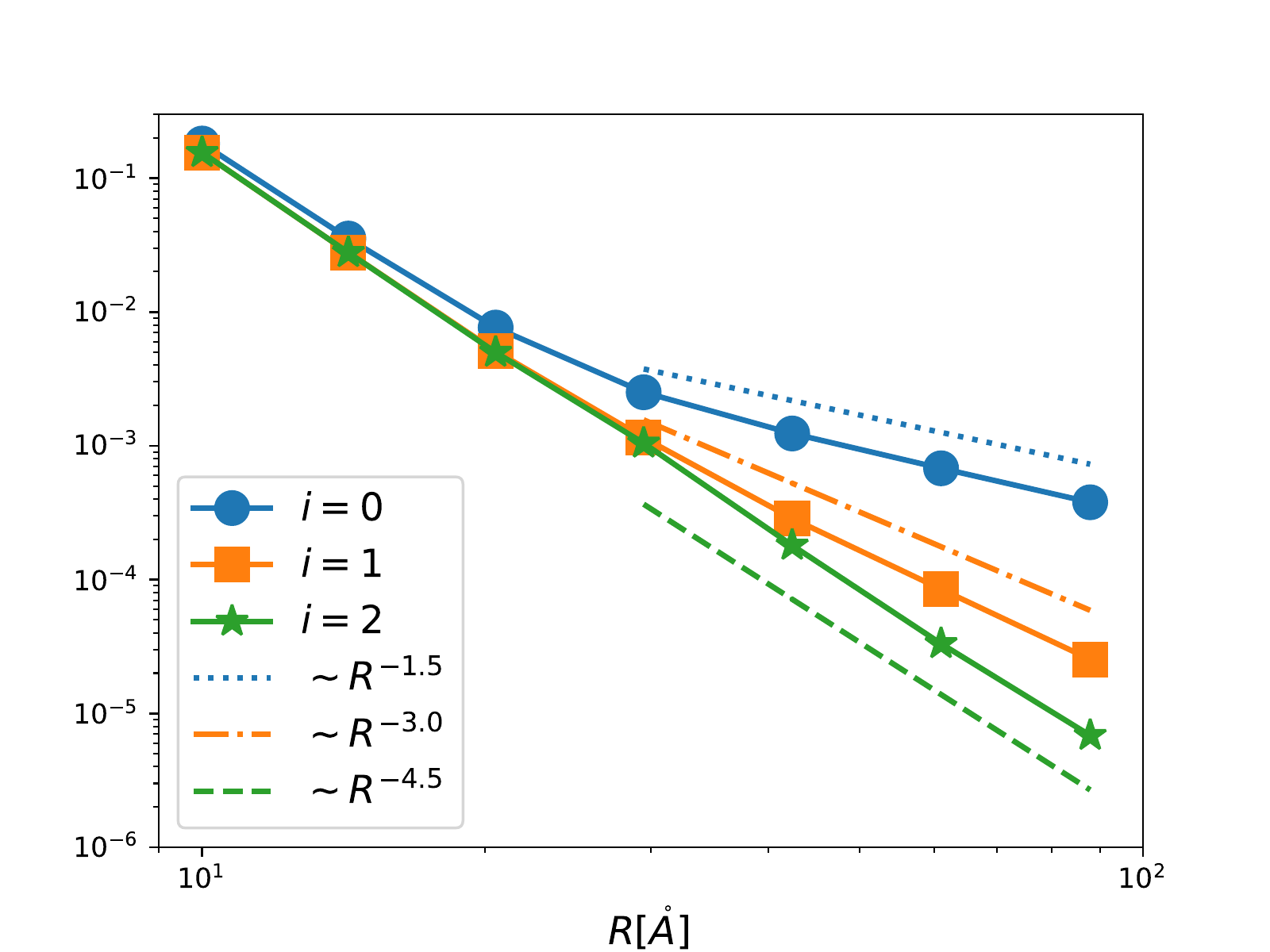}}
    \subfloat[Interstitial \label{fig:sub:conv_Du_int}]{
    \includegraphics[height=4.5cm]{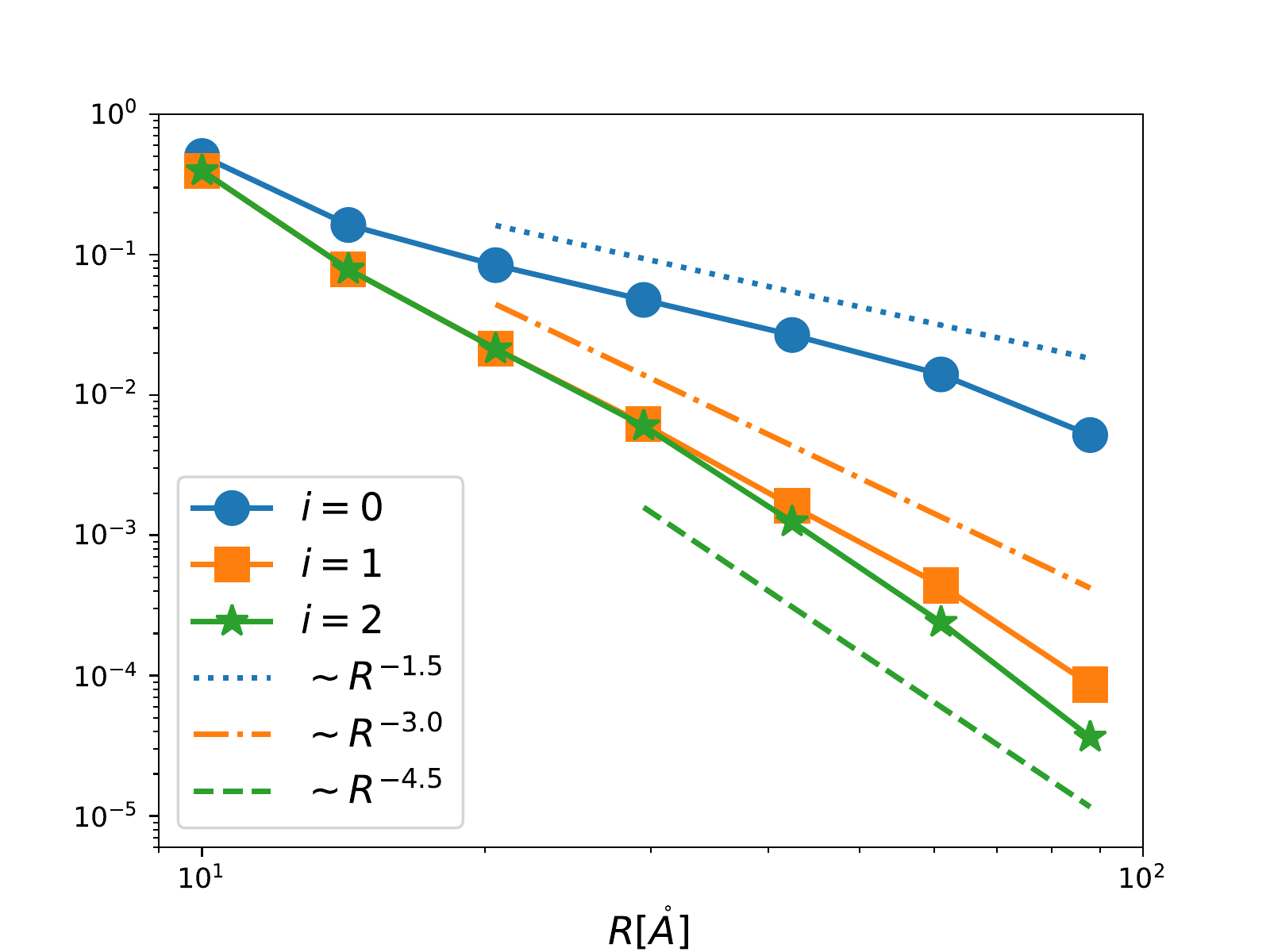}} \\
    \subfloat[Microcrack2 \label{fig:sub:conv_Du_micro2}]{
    \includegraphics[height=4.5cm]{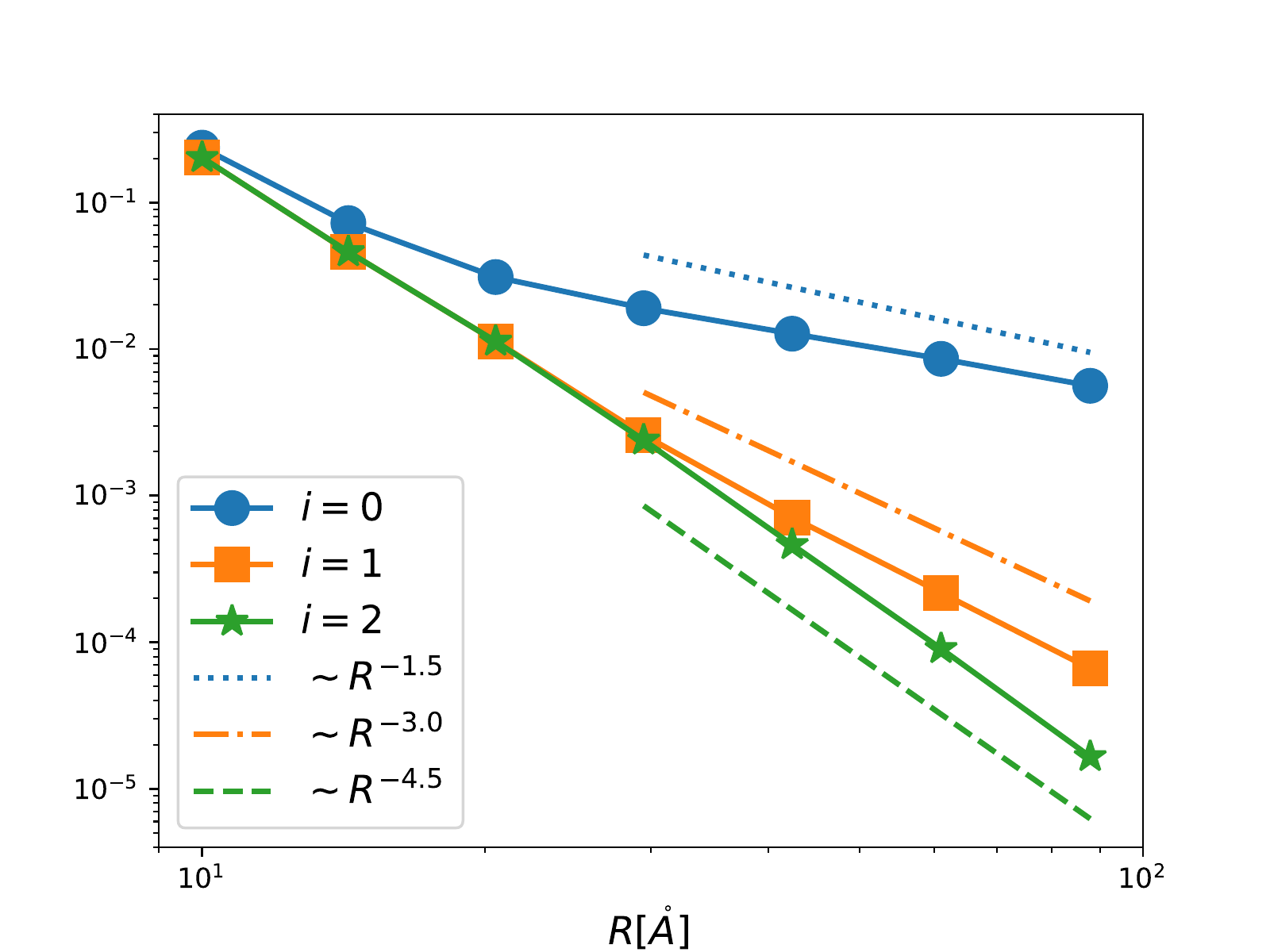}} 
    \subfloat[Microcrack3 \label{fig:sub:conv_Du_micro3}]{
    \includegraphics[height=4.5cm]{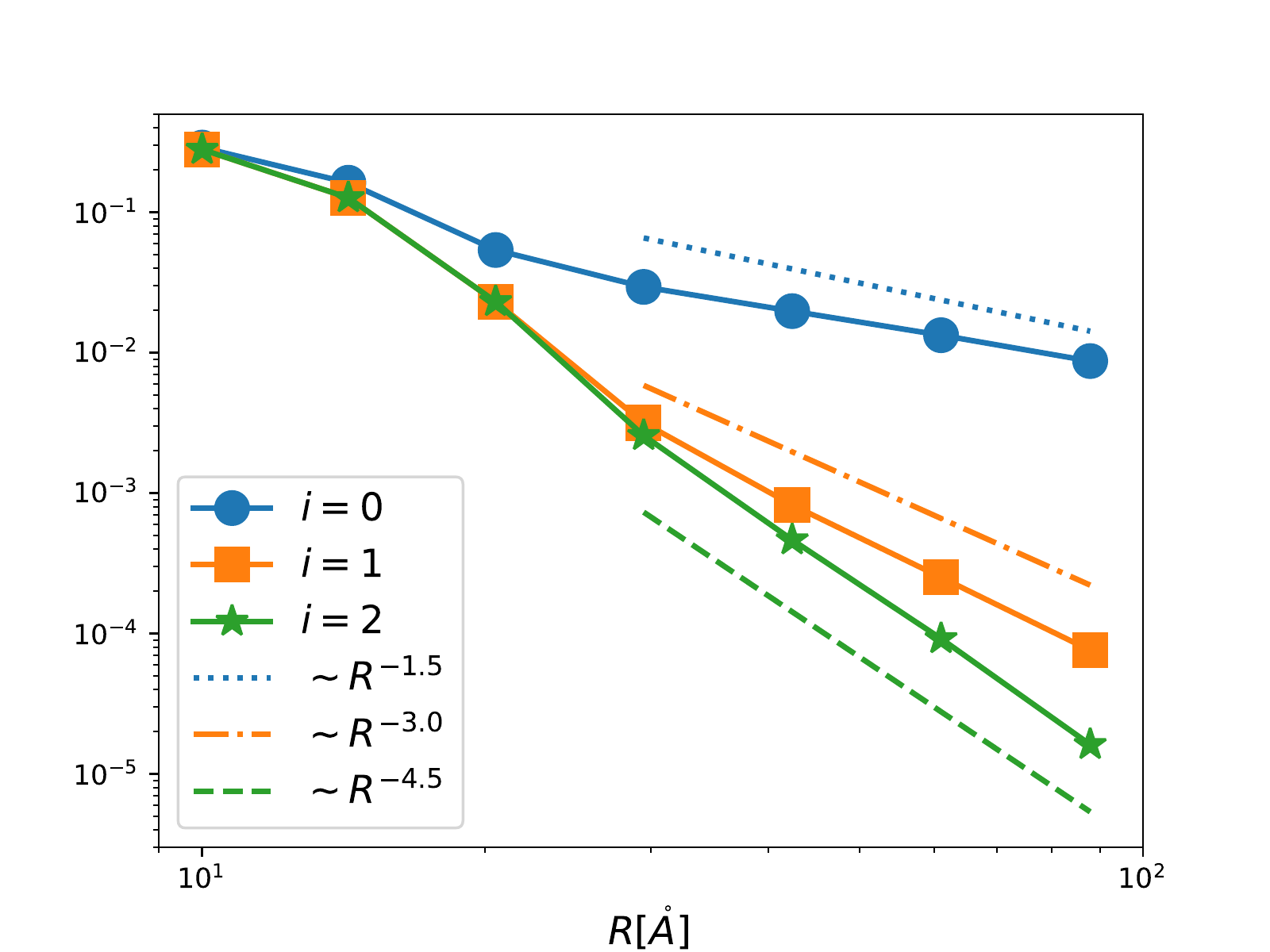}}
    \caption{Convergence of geometry error $\|D\bar{u} - D\bar{u}_{i, R}\|_{\ell^2}$ for $i=0,1,2$ against domain size $R$.}
    \label{figs:conv_Du}
\end{figure}

\subsubsection*{Energy error}

The convergence of the energy error is a natural consequence of the relationship between the geometry error and the energy error, as indicated by the quadratic relationship in \eqref{eq:GE}. Figure~\ref{figs:conv_E} plots the energy error $\big|\mathcal{E}(\bar{u})-\mathcal{E}(\bar{u}_{i, R})\big|$ for $i=0,1,2$ in the cell problems with respect to the domain size $R$. These plots further validate the predictions made in Corollary~\ref{th:galerkinbM} regarding the convergence of the energy error. It is interesting to observe that higher order boundary conditions dramatically improve the energy errors even in the pre-asymptotic regime. Since most defect simulations are often concerned primarily with accuracy energies, this is a promising result. 

\begin{figure}[!htb]
    \centering
    \subfloat[Vacancy \label{fig:sub:conv_E_singvac}]{
    \includegraphics[height=4.5cm]{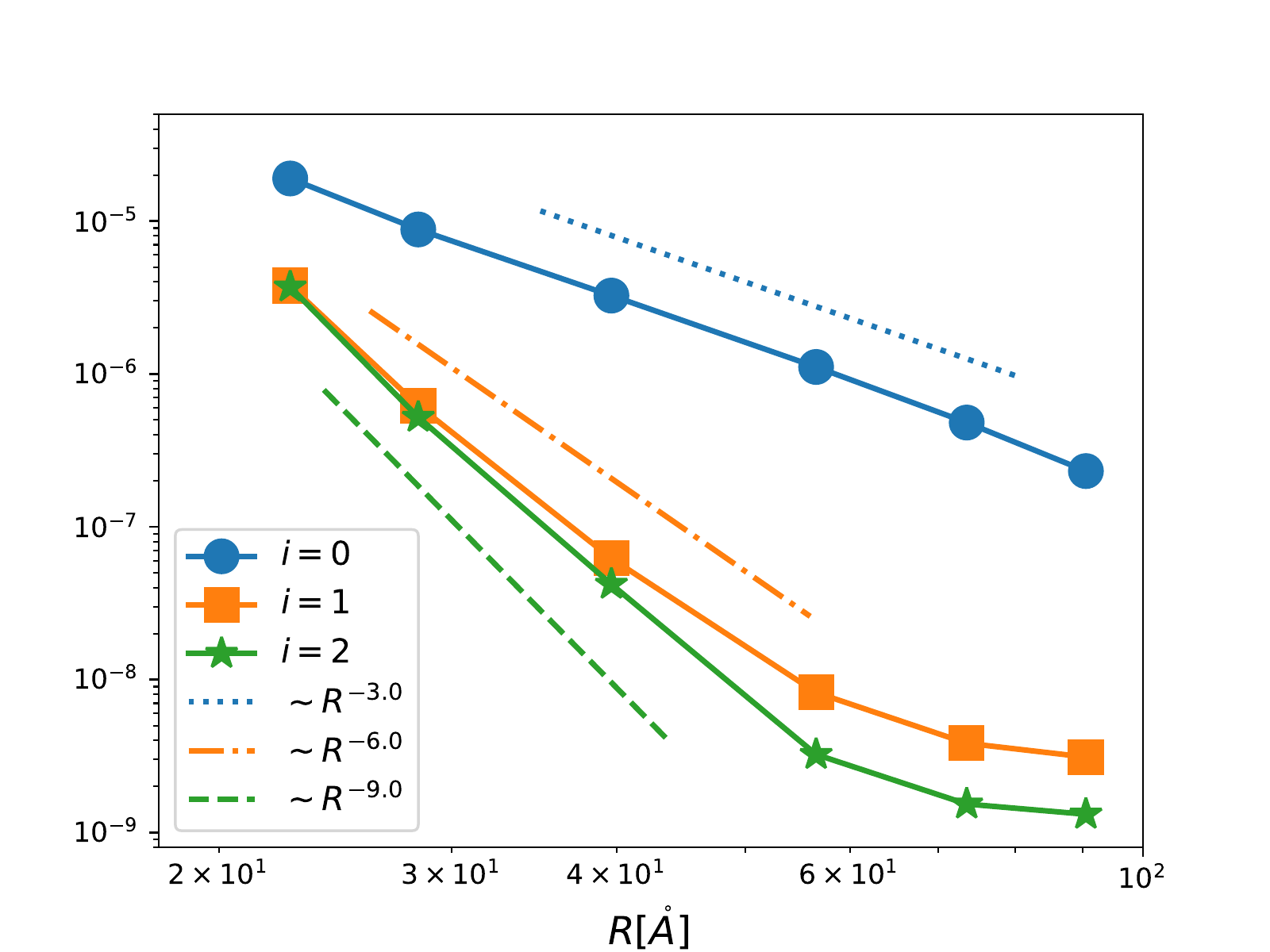}}~
    \subfloat[Divacancy \label{fig:sub:conv_E_divac}]{
    \includegraphics[height=4.5cm]{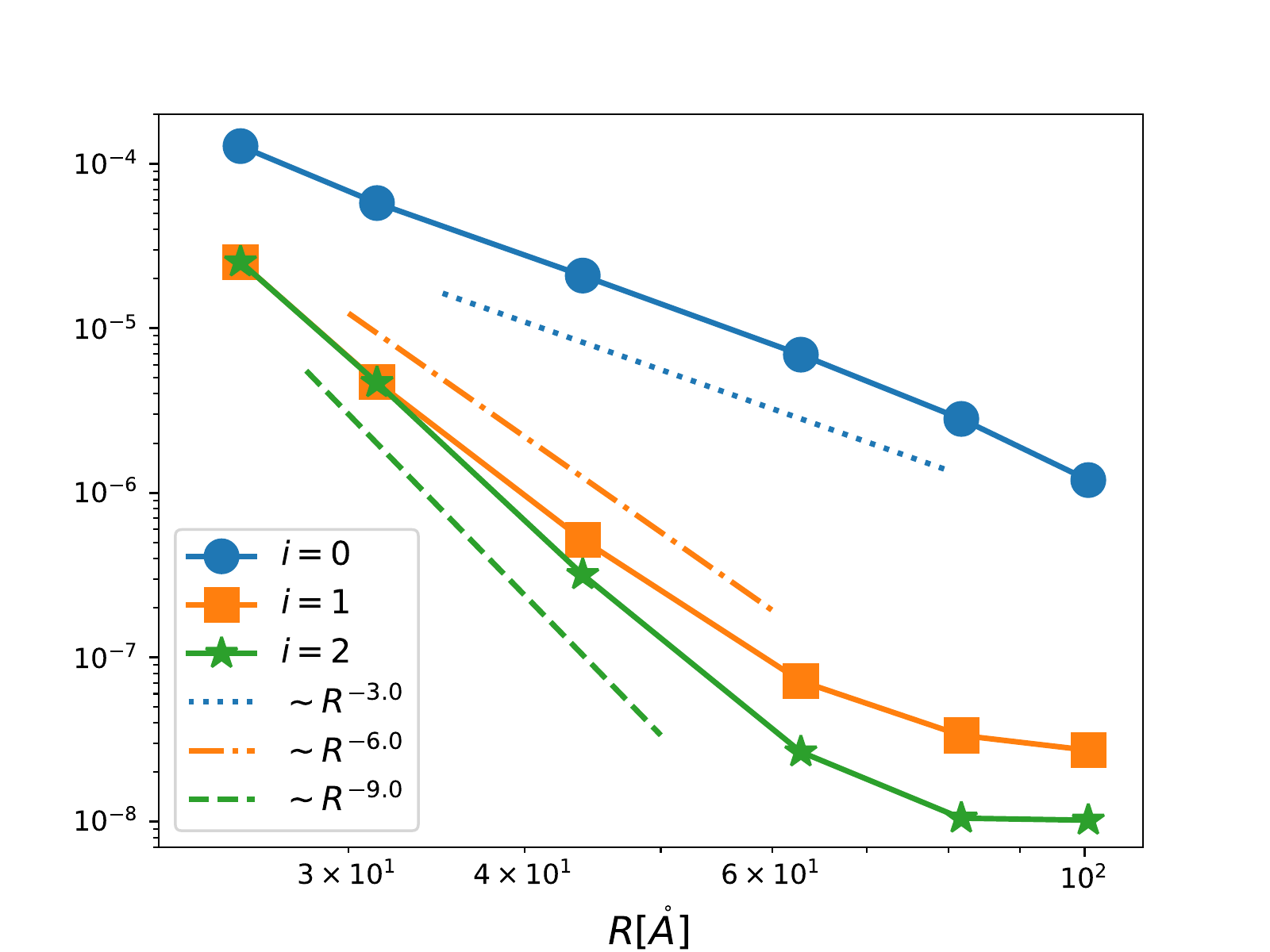}}~
    \subfloat[Interstitial \label{fig:sub:conv_E_int}]{
    \includegraphics[height=4.5cm]{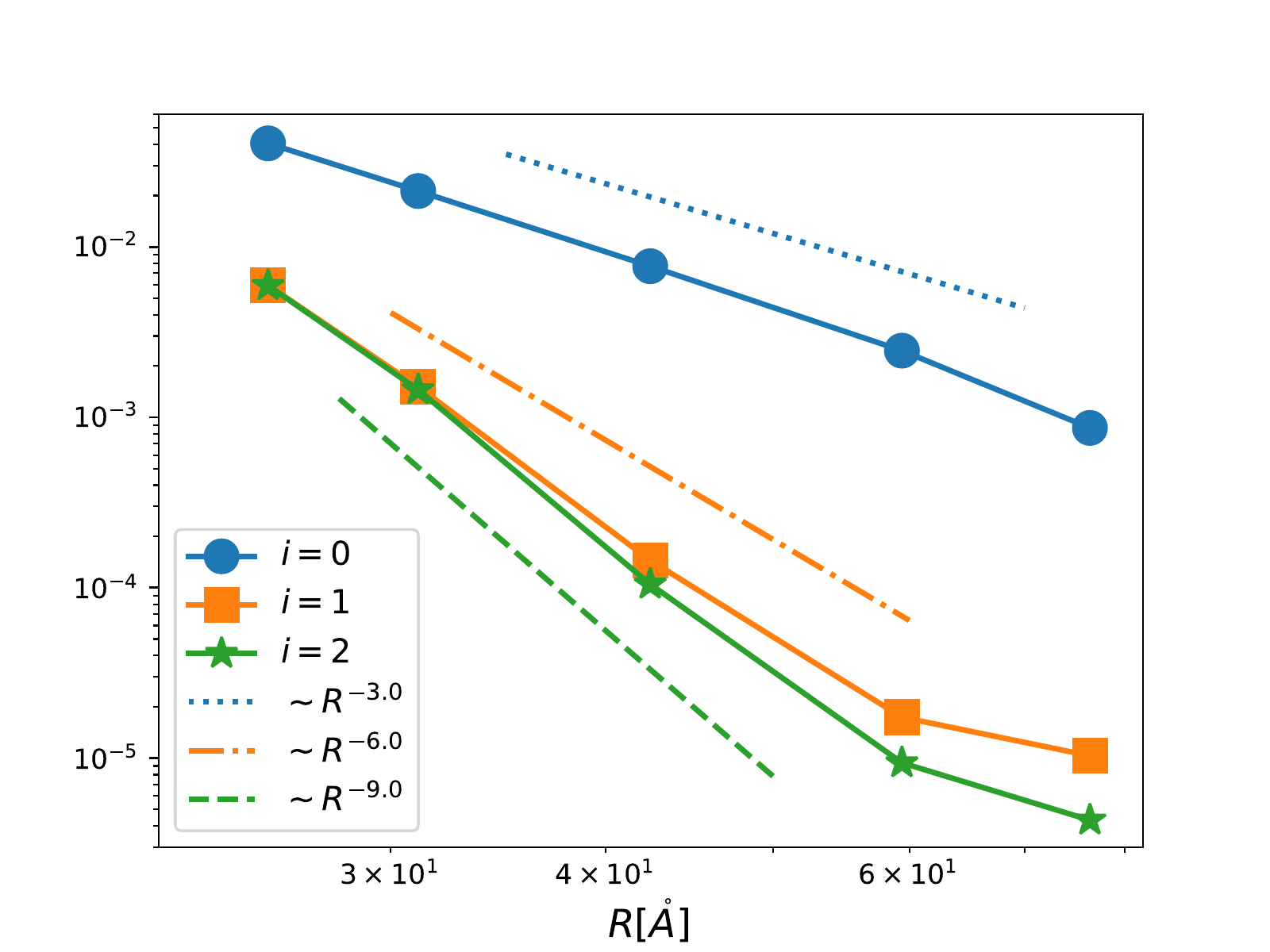}} \\
    \subfloat[Micro-crack-2 \label{fig:sub:conv_E_micro2}]{
    \includegraphics[height=4.5cm]{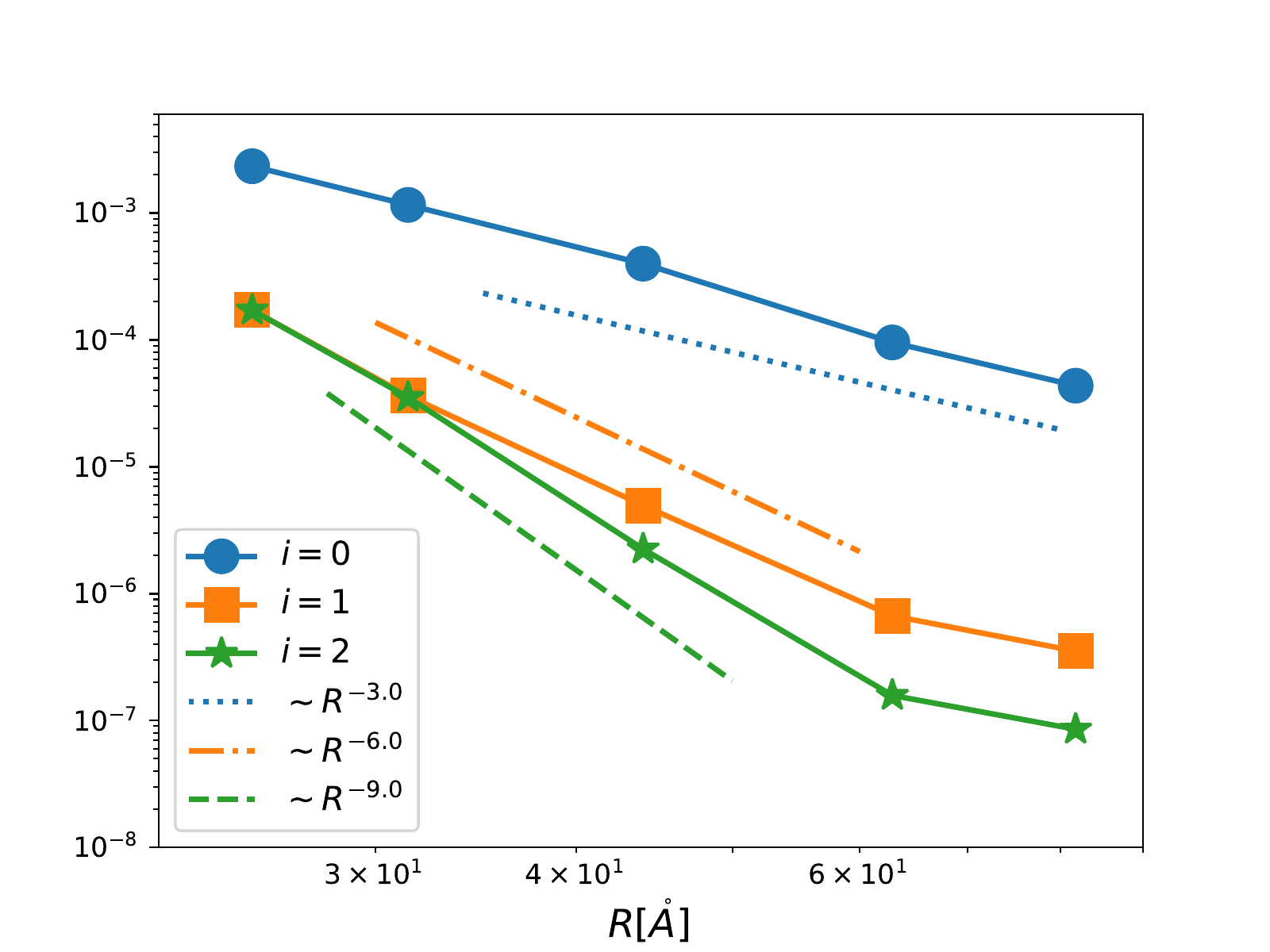}}
    \subfloat[Micro-crack-3 \label{fig:sub:conv_E_micro3}]{
    \includegraphics[height=4.5cm]{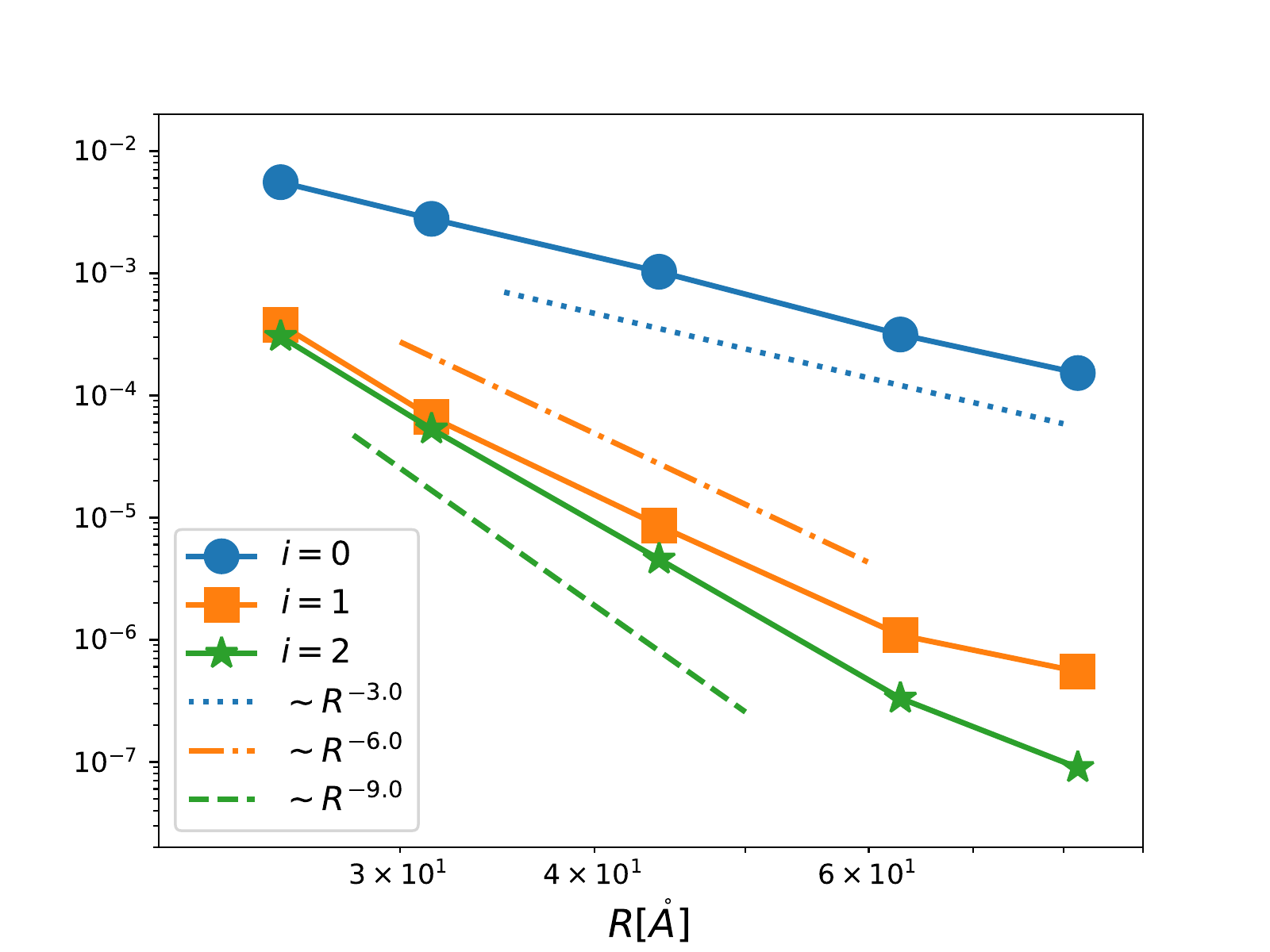}}
    \caption{Convergence of energy error $\big|\mathcal{E}(\bar{u})-\mathcal{E}(\bar{u}_{i, R})\big|$ for $i=0,1,2$ against domain size $R$.}
    \label{figs:conv_E}
\end{figure}

\section{Conclusion}
\label{sec:conclusion}
We presented a novel numerical scheme designed to improve the convergence of traditional cell approximations in defect simulations by constructing higher-order boundary conditions through multipole expansions of defect equilibrium. Our approach involves an iterative method designed to systematically improve the accuracy of approximate multipole tensor evaluations. This is complemented by rigorous error estimates. 
An interesting side-result is that our analysis also provides error estimates for the defect dipole tensor, an important quantity in materials defect modelling.
Notably, the developed approach inherently covers the anisotropic case and facilitates extensions to higher-order multipole tensors.
To evaluate the effectiveness of our approach, we presented numerical examples 
assessing both the geometry error and energy error convergence. Our results demonstrate that our numerical scheme achieves, in practice, accelerated convergence rates with respect to computational cell size. 

Our presentation here is restricted to simple lattices and point defects. Adapting the method to pure screw dislocations appears relatively straightforward (though it requires the solution of higher-order continuum PDE), but further generalisations do require additional technical difficulties to be overcome. However, there appears to be no fundamental limitation to extend the method and the results to multi-lattices and a range of other defects in some form.

\section{Proofs and Extensions}
\label{sec:proofs}

In this section, we provide proofs and extensions to complement the main numerical schemes discussed in Section~\ref{sec:moments}. These additional analyses aim to support the theoretical foundations and provide deeper insights into the key concepts and methodologies introduced in this paper.

\subsection{Proofs of the results in Section \ref{sec:moments}}
\label{sec:sub:proofs3}

In this section, we give the detailed proofs of the error estimates presented in Section~\ref{sec:moments}. First of all, we prove the Theorem~\ref{th:galerkin}.

\begin{proof}[Proof of Theorem \ref{th:galerkin}]
The existence and the error estimate for the geometry error directly follows from \cite[Theorem 3.7]{2021-defectexpansion} with $b^{(i,k)} = b^{(i,k)}_{\rm exact}$. Next, we provide the error estimate for the energy error. Since $\mathcal{E}$ is twice differentiable along the segment $\{(1-s)\bar{u}+s \bar{u}_{p,R}~|~ s\in(0,1)\}$, we have
\begin{align}
    \big|\mathcal{E}(\bar{u})-\mathcal{E}(\bar{u}_{p, R})\big| &= \Big| \int_{0}^{1} \big\<\delta \mathcal{E}\big( (1-s)\bar{u}+s \bar{u}_{p, R} \big), \bar{u}-\bar{u}_{p, R}\big\> \ds \Big| \nonumber \\[1ex]
    &\leq C M \big\| D\bar{u} - D\bar{u}_{p,R} \big\|_{\ell^2}^2,
\end{align}
where $M$ is the uniform Lipschitz constant of $\delta \mathcal{E}$. This yields the stated results.
\end{proof}

Next, we give the proof of Lemma~\ref{th:galerkinfixedb}, which follows a similar structure to that of Theorem~\ref{th:galerkin} but additionally considering the difference of moments. 

\begin{proof}[Proof of Lemma~\ref{th:galerkinfixedb}]
We first divide the target geometry error $\big\| D\bar{u} - D\bar{u}_{b,R} \big\|_{\ell^2}$ into two parts:
\begin{eqnarray}\label{eq:DubR_split}
\big\| D\bar{u} - D\bar{u}_{b,R} \big\|_{\ell^2} \leq \big\| D\bar{u} - D\bar{u}_{p,R} \big\|_{\ell^2} + \big\| D\bar{u}_{p, R} - D\bar{u}_{b,R} \big\|_{\ell^2}.
\end{eqnarray}
The first part can be directly bounded by the estimates in Theorem~\ref{th:galerkin}
\[
\big\| D\bar{u} - D\bar{u}_{p,R} \big\|_{\ell^2} \lesssim \cdot 
     R^{- d/2 - p} \cdot \log^{p+1}(R).
\]
For the second part, we note that the solutions lie in the exact and approximate moment space defined by \eqref{eq:exact_space} and \eqref{eq:approx_space}. It is shown in \cite[Lemma 6.2]{EOS2016} that lattice Green's functions satisfy $\big|D^i \mathcal{G}(\ell)\big|\leq C(1+|\ell|)^{-d-i+2}$. Hence, we can obtain
\[
\big\| D\bar{u}_{p, R} - D\bar{u}_{b,R} \big\|_{\ell^2} \lesssim \sum_{i=1}^p \lvert b^{(i,\cdot)} - b_{\rm exact}^{(i,\cdot)} \rvert \cdot \Big(\sum_{|\ell|>R}\big|D^i\mathcal{G}(\ell)\big|^2\Big)^{1/2} \lesssim \sum_{i=1}^p \lvert b^{(i,\cdot)} - b_{\rm exact}^{(i,\cdot)} \rvert \cdot R^{1-d/2-i}.
\]
Taking into account the two estimates above with the decomposition \eqref{eq:DubR_split}, one can yield the stated results.
\end{proof}

We are ready to prove Theorem~\ref{thm:moments}. Our goal is to analyze the moment error $\big \lvert \mathcal{I}_{i,R}[\bar{u}_{b}] - \mathcal{I}_{i}[\bar{u}] \big \rvert$ for $i=1,\ldots,p$ in terms of the domain size $R$ for the purpose of estimating the unspecified errors $\big \lvert b^{(i,\cdot)}-b_{\rm exact}^{(i,\cdot)}\big \rvert$ shown in Lemma~\ref{th:galerkinfixedb}.  

\begin{proof}[Proof of Theorem~\ref{thm:moments}]
We first split the moment error into two parts:
\begin{eqnarray}\label{eq:sp_Idiff}
\big \lvert \mathcal{I}_{i,R}[\bar{u}_{b,R}] - \mathcal{I}_{i}[\bar{u}] \big \rvert \leq \big\lvert \mathcal{I}_{i,R}[\bar{u}] - \mathcal{I}_{i}[\bar{u}] \big\rvert + \big \lvert \mathcal{I}_{i,R}[\bar{u}_{b,R}] - \mathcal{I}_{i,R}[\bar{u}] \big \rvert =: I_1 + I_2.
\end{eqnarray}

To estimate $I_1$, after a straightforward manipulation, we can obtain
\begin{align}\label{eq:I1}
\big\lvert \mathcal{I}_{i,R}[\bar{u}] - \mathcal{I}_{i}[\bar{u}] \big\rvert &\lesssim \sum_{\lvert \ell \rvert > R/3} \lvert \ell \rvert^{-d-p-1+i} \cdot \log^{p}(\lvert \ell \rvert)\nonumber \\[1ex]
&\lesssim R^{i-p-1} \cdot \log^{p}(R),
\end{align}
where the last inequality follows from the estimates in \cite[Lemma 7.1]{2021-defectexpansion}.

For the term $I_2$, note that both $\bar{u}$ and $\bar{u}_{b,R}$ solve the discrete force equilibrium equations for $\lvert \ell \rvert < 2R/3$ (actually $\lvert \ell \rvert<R-R_0$ is already sufficient), hence we have
\begin{align}\label{eq:Hdiff}
H [\bar{u}_{b,R}](\ell) - H [\bar{u}](\ell) = \int_0^1 (1-t)\Div \big( &\nabla^3 V(tD\bar{u}_{b,R})[D\bar{u}_{b,R}, D\bar{u}_{b,R}] \nonumber \\
&- \nabla^3 V(tD\bar{u})[D\bar{u}, D\bar{u}]\big) \dt.
\end{align}
Before inserting the above term into the moment sum \eqref{eq:results:defn_Ij}, we first estimate the stress difference in \eqref{eq:Hdiff}. Note that for any bounded Lipschitz function $f$ and any sub-multiplicative norm and product:
\[
\big\lvert f(z) \cdot z \cdot z - f(x) \cdot x \cdot x \big\rvert \lesssim \lvert z-x \rvert (\lvert x \rvert + \lvert z-x \rvert) (1+\lvert x \rvert + \lvert z-x \rvert). 
\]
Hence, the stress difference in \eqref{eq:Hdiff} can be further estimated by
\begin{align*}
&\big\lvert \nabla^3 V(tD\bar{u}_{b,R}) [D\bar{u}_{b,R}, D\bar{u}_{b,R}] - \nabla^3 V(tD\bar{u})[D\bar{u}, D\bar{u}] \big\rvert \\[1ex]
\lesssim& \lvert D\bar{u}_{b,R} - D\bar{u} \rvert \cdot \big(\lvert D\bar{u} \rvert+\lvert D\bar{u}_{b,R} - D\bar{u} \rvert\big) \cdot \big(1+\lvert D\bar{u} \rvert+\lvert D\bar{u}_{b,R} - D\bar{u} \rvert\big) \\[1ex]
\lesssim& \lvert D\bar{u}_{b,R} - D\bar{u} \rvert \cdot \big(\lvert \ell \rvert^{-d}+\lvert D\bar{u}_{b,R} - D\bar{u} \rvert\big),
\end{align*}
where the last inequality follows from the decay estimate of the equilibrium for point defects~\cite[Theorem 1]{EOS2016}.
Inserting that into the moment sum after partially summing, we obtain
\begin{align}\label{eq:I2}
    &\big\lvert \mathcal{I}_{i,R}[\bar{u}_{b,R}] - \mathcal{I}_{i,R}[\bar{u}] \big\rvert \nonumber \\[1ex]
    \lesssim&~\int_0^1 (1-t) \sum_{\ell \in \Lambda \cap B_{2R/3}({\bf 0})} \lvert \ell \rvert^{i-1} \cdot \big\lvert \nabla^3 V(tD\bar{u}_{b,R}) [D\bar{u}_{b,R}, D\bar{u}_{b,R}] - \nabla^3 V(tD\bar{u})[D\bar{u}, D\bar{u}] \big\rvert \dt \nonumber \\[1ex]
    \lesssim&~\sum_{\ell \in \Lambda \cap B_{2R/3}({\bf 0})} \lvert \ell \rvert^{i-1} \cdot \lvert D\bar{u}_{b,R} - D\bar{u} \rvert \cdot \big(\lvert \ell \rvert^{-d}+\lvert D\bar{u}_{b,R} - D\bar{u} \rvert\big) \nonumber \\[0.5ex]
    \lesssim&~\lVert D\bar{u}_{b,R} - D\bar{u} \rVert_{\ell^2} \cdot \Bigg(R^{i-1}\lVert D\bar{u}_{b,R} - D\bar{u} \rVert_{\ell^2} + \Bigg(\sum_{\ell \in \Lambda \cap B_{2R/3}({\bf 0})} \lvert \ell \rvert^{i-2-2d}\Bigg)^{1/2}\Bigg),
\end{align}
where the last term can be further estimated by
\begin{equation*}
    \Bigg(\sum_{\ell \in \Lambda \cap B_{2R/3}({\bf 0})} \lvert \ell \rvert^{2i-2-2d}\Bigg)^{1/2} \lesssim \alpha_i(R) := \begin{cases}
          1 &~i<1+d/2 \\
          \log(R) &~i=1+d/2 \\
          R^{i-1-d/2} &~i>1+d/2 
       \end{cases}.
\end{equation*}
Combing \eqref{eq:sp_Idiff}, \eqref{eq:I1} with \eqref{eq:I2}, we obtain
\begin{equation}\label{eq:momenterror_}
    \big\lvert \mathcal{I}_{i,R}[\bar{u}_{b,R}] - \mathcal{I}_{i}[\bar{u}] \big\rvert \lesssim R^{i-p-1} \log^p(R) + R^{i-1} \big\lVert D\bar{u}_{b,R} - D\bar{u} \rVert_{\ell^2}^2 + \alpha_i(R) \lVert D\bar{u}_{b,R} - D\bar{u} \rVert_{\ell^2},
\end{equation}
which yields the stated result.
\end{proof}

\subsection{The relationship between force moments, discrete and continuous coefficients}
\label{sec:sub:abIrelation}

In this section, we establish the explicit relationship between force moments and coefficients in both discrete and continuous expansions (cf.~\eqref{eq:exp_b} and \eqref{eq:exp_a}). We focus on the case of three-dimensional systems ($d=3$) and consider up to the third order expansions ($p=3$). 

Applying Theorem~\ref{thm:pointdef} and Lemma~\ref{thm:structurewithmomentscont}, the decomposition of the equilibrium $\bar{u}$ in both discrete and continuous versions can be written as 
\begin{align}\label{eq:decomp_u_3}
\bar{u} &=~\sum_{k=1}^3 b^{(1,k)}:D_{\mathcal{S}} (\mathcal{G})_{\cdot k} + \sum_{k=1}^3 b^{(2,k)}:D^2_{\mathcal{S}} (\mathcal{G})_{\cdot k} + \sum_{k=1}^3 b^{(3,k)}:D^3_{\mathcal{S}} (\mathcal{G})_{\cdot k} + O(\lvert \ell \rvert^{-5}), \nonumber \\[1ex]
&=~\Bigg( \sum_{k=1}^3 a^{(1,0,k)}:\nabla (G_0)_{\cdot k} + \sum_{k=1}^3 a^{(1,1,k)}:\nabla (G_1)_{\cdot k}\Bigg) + \sum_{k=1}^3 a^{(2,0,k)}:\nabla^2 (G_0)_{\cdot k} \nonumber \\
&\quad + \sum_{k=1}^3 a^{(3,0,k)}:\nabla^3 (G_0)_{\cdot k} + O(\lvert \ell \rvert^{-5}) \nonumber \\[1ex]
&=:~u_1 + u_2 + u_3 + O(\lvert \ell \rvert^{-5}).
\end{align}

\subsubsection{The relationship between the continuous coefficients $a^{(i,n,k)}$ and the discrete coefficients $b^{(i,k)}$} In order to derive the identities shown in \eqref{eq:a-b}, we first Taylor expand the discrete difference stencil as well as its higher-order discrete difference at site $\ell$, for any $\rho,\sigma,\tau \in \mathcal{R}$,
\begin{align*}
  D_\rho \mathcal G &= \nabla_\rho G_0 + \nabla_\rho G_1 + \frac{1}{2} \nabla^2_{\rho\rho} G_0 + \frac{1}{6} \nabla^3_{\rho\rho\rho} G_0 + {\rm h.o.t.}, \\[1ex]
  D^2_{\rho\sigma} \mathcal G &= \nabla^2_{\rho\sigma} G_0 + \frac{1}{2} \nabla^3_{\rho\sigma\sigma} G_0 +  \frac{1}{2} \nabla^3_{\rho\rho\sigma} G_0 + {\rm h.o.t.}, \\[1ex]
  D^3_{\rho\sigma\tau} \mathcal G &= \nabla^3_{\rho\sigma\tau} G_0 +  {\rm h.o.t.}
\end{align*}

 For the first term $u_1$ in \eqref{eq:decomp_u_3}, by matching the first-order terms and applying the above identities, we can obtain
\begin{align*}
&\sum_{j,k=1}^3 \Big( (a^{(1,0,k)})_{\cdot j} \cdot \partial_j (G_0)_{\cdot k} + (a^{(1,1,k)})_{\cdot j} \cdot \partial_j (G_1)_{\cdot k} \Big) \\[1ex]
=~&\sum_{k=1}^3 \sum_{\rho\in\mathcal{R}} (b^{(1,k)})_\rho \cdot D_\rho \mathcal{G}_{\cdot k} + {\rm h.o.t.}\\[1ex]
=~&\sum_{j,k=1}^3 \sum_{\rho\in\mathcal{R}} \Big( (b^{(1,k)})_\rho \cdot \partial_j (G_0)_{\cdot k} \cdot \rho_j + (b^{(1,k)})_\rho \cdot \partial_j (G_1)_{\cdot k} \cdot \rho_j \Big). 
\end{align*}
Hence, for each $j,k=1,2,3$,
\begin{eqnarray}\label{a-b-1}
(a^{(1,0,k)})_{\cdot j} = (a^{(1,1,k)})_{\cdot j} = \sum_{\rho\in\mathcal{R}} (b^{(1,k)})_\rho \cdot \rho_j.
\end{eqnarray}

To get the second-order coefficients ($i=2$), it is necessary to consider the term $u_2$ as well as the square term in $u_1$, that is,
\begin{align*}
\sum_{j,m,k=1}^3 (a^{(2,0,k)})_{\cdot jm} \cdot \partial^2_{jm} (G_0)_{\cdot k} =&~\sum_{k=1}^3 \sum_{\rho, \sigma \in \mathcal{R}} (b^{(2,k)})_{\rho\sigma} \cdot D^2_{\rho\sigma} \mathcal{G}_{\cdot k} +  \sum_{k=1}^3 \sum_{\rho\in\mathcal{R}} (b^{(1,k)})_\rho \cdot  \Big(\frac{1}{2} \nabla^2_{\rho\rho} G_0 \Big)_{\cdot k} + {\rm h.o.t.} \\[1ex]
=&~\sum_{j,m,k=1}^3 \sum_{\rho, \sigma \in \mathcal{R}} (b^{(2,k)})_{\rho\sigma} \cdot \partial^2_{jm} (G_0)_{\cdot k} \cdot \rho_j \sigma_m \\
&+ \frac{1}{2} \sum_{j,m,k=1}^3 \sum_{\rho\in\mathcal{R}} (b^{(1,k)})_\rho \cdot \partial^2_{jm} (G_0)_{\cdot k} \cdot \rho_j\rho_m.
\end{align*}
Hence, for each $j, m, k=1,2,3$, we have
\begin{align}\label{a-b-2}
(a^{(2,0,k)})_{\cdot jm} =& \sum_{\rho, \sigma \in \mathcal{R}} (b^{(2,k)})_{\rho\sigma} \cdot \rho_j \sigma_m + \frac{1}{2} \sum_{\rho\in\mathcal{R}} (b^{(1,k)})_\rho \cdot \rho_j\rho_m \nonumber \\[1ex]
=&  \sum_{\rho, \sigma \in \mathcal{R}} (b^{(2,k)})_{\rho\sigma} \cdot \rho_j \sigma_m + \frac{1}{2} \sum_{\rho\in \mathcal{R}} (a^{(1,0,k)})_{\cdot j} \cdot \rho_m .
\end{align}

Analogously, after a tedious calculation, we obtain the following identity
\begin{align*}
  &\sum_{j,m,n,k=1}^3 (a^{(3,0,k)})_{\cdot jmn} \cdot \partial^3_{jmn} (G_0)_{\cdot k} \\[1ex]
  =&~\sum_{k=1}^3 \sum_{\rho,\sigma,\tau \in \mathcal{R}} (b^{(3,k)})_{\rho\sigma\tau} \cdot D^3_{\rho\sigma\tau} \mathcal{G}_{\cdot k} 
  + \sum_{k=1}^3 \sum_{\rho,\sigma\mathcal{R}} (b^{(2,k)})_{\rho\sigma} \cdot \Big(\frac{1}{2} \nabla^3_{\rho\sigma\sigma} G_0 + \frac{1}{2} \nabla^3_{\rho\rho\sigma} G_0 \Big)_{\cdot k} \\
  & + \sum_{k=1}^3 \sum_{\rho\in\mathcal{R}} (b^{(1,k)})_\rho \cdot \Big(\frac{1}{6} \nabla^3_{\rho\rho\rho} G_0 \Big)_{\cdot k} + {\rm h.o.t.}  \\[1ex]
  =& \sum_{j,m,n,k=1}^3 \sum_{\rho,\sigma,\tau \in\mathcal{R}} (b^{(3,k)})_{\rho\sigma\tau} \cdot \partial^3_{jmn} (G_0)_{\cdot k} \cdot \rho_j \sigma_m \tau_n\\
  & + \frac{1}{2} \sum_{j,m,n,k=1}^3 \sum_{\rho,\sigma\in\mathcal{R}} (b^{(2,k)})_{\rho\sigma} \cdot \partial^3_{jmn} (G_0)_{\cdot k} \cdot (\rho_j\sigma_m\sigma_n + \rho_j\rho_m\sigma_n) \\  
  &+ \frac{1}{6} \sum_{j,m,n,k=1}^3 \sum_{\rho\in\mathcal{R}} (b^{(1,k)})_\rho \cdot \partial^3_{jmn} (G_0)_{\cdot k} \cdot \rho_j\rho_m\rho_n.
  \end{align*}
Therefore, if we denote $a^{(3,0,k)}$ as a collection of $(a^{(3,0,k)})_{\cdot jmn}$ for all $j,m,n=1,2,3$, then for each $k=1,2,3$, we can obtain 
\begin{align}\label{a-b-3}
  a^{(3,0,k)} =& \sum_{\rho,\sigma,\tau \in\mathcal{R}} (b^{(3,k)})_{\rho\sigma\tau} \cdot \rho \otimes \sigma \otimes \tau + \sum_{\rho,\sigma \in \mathcal{R}} (b^{(2,k)})_{\rho\sigma} \cdot \rho \odot \sigma \odot \sigma \nonumber \\
  &+ \frac{1}{6} \sum_{\rho \in \mathcal{R}} (b^{(1,k)})_{\rho} \cdot \rho \otimes \rho \otimes \rho.
\end{align}
  
Taking into account \eqref{a-b-1}, \eqref{a-b-2} with \eqref{a-b-3}, we can acquire the relationship between $a^{(i,n,k)}$ and $b^{(i,k)}$ presented in \eqref{eq:a-b}.

\subsubsection{The relation between the force moments $\mathcal{I}_i[\bar{u}]$ and the discrete coefficients $b^{(i,k)}$} Next, we give a concrete formulation of \eqref{eq:bIrelation}. From the definition of force moments, for $i=1,2,3$, we have
\begin{eqnarray}\label{eq:I}
\mathcal{I}_i[\bar{u}] = \sum_{\ell \in \La} H [\bar{u}](\ell) \otimes \ell^{\otimes i}.
\end{eqnarray}
The relation between moments and discrete coefficients is derived based on \eqref{eq:I}. As a matter of fact, recalling the decomposition \eqref{eq:decomp_u_3} and the fact that $H[\mathcal{G}_{\cdot k}](\ell) = \delta_{0,\ell} e_k$, we can obtain
\begin{align}\label{eq:bI}
H[u_{1}] =& H\Bigg[\sum_{k=1}^3 b^{(1,k)}:D_{\mathcal{S}} (\mathcal{G})_{\cdot k}\Bigg] = \sum_{k=1}^3 \sum_{\rho\in\mathcal{R}} b^{(1,k)}_{\rho} \cdot D_{\rho}\delta_0  e_k, \nonumber \\[1ex]
H[u_{2}] =& H\Bigg[\sum_{k=1}^3 b^{(2,k)}:D^2_{\mathcal{S}} (\mathcal{G})_{\cdot k}\Bigg] = \sum_{k=1}^3 \sum_{\rho, \sigma \in \mathcal{R}} (b^{(2,k)})_{\rho\sigma} \cdot D^2_{\rho\sigma} \delta_0  e_k, \\[1ex]
H[u_{3}] =& H\Bigg[\sum_{k=1}^3 b^{(3,k)}:D^3_{\mathcal{S}} (\mathcal{G})_{\cdot k}\Bigg] = \sum_{k=1}^3 \sum_{\rho, \sigma, \tau \in \mathcal{R}} (b^{(2,k)})_{\rho\sigma\tau} \cdot D^3_{\rho\sigma\tau} \delta_0 e_k. \nonumber
\end{align}
Hence, combining \eqref{eq:I} with \eqref{eq:bI}, one can acquire that
\begin{align}\label{eq:b-I-1}
\mathcal{I}_1[\bar{u}] &= \sum_{k=1}^3 \sum_{\rho \in \mathcal{R}}\sum_{\ell\in\Lambda} (b^{(1,k)})_{\rho} \cdot D_{\rho} \delta_0 e_k \otimes \ell \nonumber  \\[1ex]
&= - \sum_{k=1}^3 \sum_{\rho\in\mathcal{R}} (b^{(1,k)})_{\rho} \cdot e_k \otimes \rho \nonumber  \\[1ex]
&= - \sum_{j,k=1}^3 \sum_{\rho\in\mathcal{R}} (b^{(1,k)})_{\rho} \cdot \rho_j \cdot e_k \otimes e_j,
\end{align}
which yields the relation between the first order coefficient $b^{(1,\cdot)}$ and dipole moment $\mathcal{I}_1[\bar{u}]$.

As for the tripole moment $\mathcal{I}_2[\bar{u}]$, similarly we have
\begin{align}\label{eq:b-I-2}
\mathcal{I}_2[\bar{u}] =& \sum_{k=1}^3 \sum_{\rho\in\mathcal{R}}\sum_{\ell\in\Lambda} (b^{(1,k)})_{\rho} \cdot D_{\rho} \delta_0  e_k \otimes \ell \otimes \ell + \sum_{k=1}^3 \sum_{\rho, \sigma \in \mathcal{R}} \sum_{\ell \in \Lambda} (b^{(2,k)})_{\rho\sigma} \cdot D^2_{\rho\sigma} \delta_0  e_k \otimes \ell \otimes \ell \nonumber \\[1ex]
=& \sum_{k=1}^3 \sum_{\rho\in\mathcal{R}} (b^{(1,k)})_{\rho} e_k \otimes \rho \otimes \rho + 2 \sum_{k=1}^3 \sum_{\rho, \sigma \in\mathcal{R}} (b^{(2,k)})_{\rho\sigma} \cdot e_k \otimes \rho \otimes \sigma  \nonumber \\[1ex] 
=& \sum_{j,m,k=1}^3 \sum_{\rho\in\mathcal{R}} (b^{(1,k)})_{\rho} \cdot \rho_j \rho_m e_k \otimes e_j \otimes e_m + 2\sum_{j,m,k=1}^3 \sum_{\rho, \sigma \in \mathcal{R}} (b^{(2,k)})_{\rho\sigma} \cdot \rho_j \sigma_m \cdot e_k \otimes e_j \otimes e_m.
\end{align}
For the quadrupole moment $\mathcal{I}_3[\bar{u}]$, after a tedious calculation, we can obtain
\begin{align}\label{eq:b-I-3}
\mathcal{I}_3[\bar{u}] =& \sum_{k=1}^3 \sum_{\rho \in \mathcal{R}}\sum_{\ell\in\Lambda} (b^{(1,k)})_{\rho} \cdot D_{\rho} \delta_0 e_k \otimes \ell \otimes \ell \otimes \ell + \sum_{k=1}^3 \sum_{\rho, \sigma\in\mathcal{R}} \sum_{\ell\in\mathcal{R}} (b^{(2,k)})_{\rho\sigma}\cdot D^2_{\rho\sigma} \delta_0  e_k \otimes \ell \otimes \ell \otimes \ell  \nonumber \\
 &+ \sum_{k=1}^3 \sum_{\rho, \sigma, \tau\in \mathcal{R}} \sum_{\ell\in\Lambda} (b^{(3,k)})_{\rho\sigma\tau} \cdot D^3_{\rho\sigma\tau} \delta_0  e_k \otimes \ell \otimes \ell \otimes \ell \nonumber \\[1ex]
 =& \sum_{k=1}^3 \sum_{\rho\in\mathcal{R}} \sum_{\ell\in\Lambda}(b^{(1,k)})_{\rho} \cdot e_k \otimes D_{-\rho} \Big( \ell \otimes \ell \otimes \ell\Big) \nonumber \\
 &+ \sum_{k=1}^3 \sum_{\rho, \sigma\in\mathcal{R}} \sum_{\ell\in\Lambda} (b^{(2,k)})_{\rho\sigma} \cdot e_k \otimes D_{-\rho}D_{-\sigma} \Big(\ell \otimes \ell \otimes \ell \Big)  \nonumber \\
 &+ \sum_{k=1}^3 \sum_{\rho, \sigma, \tau\in\mathcal{R}} \sum_{\ell\in\Lambda} (b^{(3,k)})_{\rho\sigma\tau}\cdot e_k \otimes D_{-\rho}D_{-\sigma}D_{-\tau}\Big( \ell \otimes \ell \otimes \ell \Big) \nonumber \\[1ex]
 =& -\sum_{k=1}^3 \sum_{\rho\in\mathcal{R}} (b^{(1,k)})_{\rho} \cdot e_k \otimes \rho \otimes \rho \otimes \rho  - 6 \sum_{k=1}^3 \sum_{\rho, \sigma\in\mathcal{R}} (b^{(2,k)})_{\rho\sigma} \cdot e_k \otimes \rho \odot \sigma \odot \sigma  \nonumber \\ 
&- 6 \sum_{k=1}^3 \sum_{\rho, \sigma, \tau\in\mathcal{R}} (b^{(3,k)})_{\rho\sigma\tau} \cdot e_k \otimes \rho \odot \sigma \odot \tau.
\end{align}
Hence, we yield the results in \cite[Lemma 5.6]{2021-defectexpansion} or \eqref{eq:bIrelation}.

\subsubsection{The relationship between the force moments $I_i[\bar{u}]$ and the continuous coefficients $a^{(i,n,k)}$} Combining the results presented above, we then derive the relation between the moments $I_i[\bar{u}]$ and the continuous coefficients $a^{(i,n,k)}$. More precisely, for the dipole moment, by inserting \eqref{a-b-1} and \eqref{a-b-2} into \eqref{eq:b-I-1}, we have
\[
\mathcal{I}_1[\bar{u}] = - \sum_{j,k=1}^3 \sum_{\rho\in\mathcal{R}} (b^{(1,k)})_{\rho} \cdot \rho_j \cdot e_k \otimes e_j = - a^{(1,0,\cdot)} = - a^{(1,1,\cdot)}.
\]
Similarly, taking into account \eqref{a-b-2} with \eqref{eq:b-I-2}, we get
\begin{align*}
\mathcal{I}_2[\bar{u}] =& \sum_{j,m,k=1}^3 \sum_{\rho\in\mathcal{R}} (b^{(1,k)})_{\rho} \cdot \rho_j \rho_m \cdot e_k \otimes e_j \otimes e_m  \nonumber \\
&+ 2\sum_{j,m,k=1}^3 \sum_{\rho, \sigma\in\mathcal{R}} (b^{(2,k)})_{\rho\sigma} \cdot \rho_j \sigma_m \cdot e_k \otimes e_j \otimes e_m \nonumber \\[1ex]
=& \sum_{j,m,k=1}^3 \Big(\sum_{\rho\in\mathcal{R}} (a^{(1,0,k)})_{\cdot j} \cdot \rho_m \Big) e_k \otimes e_j \otimes e_m \nonumber \\
&+ \sum_{j,m,k=1}^3  \Big( 2 (a^{(2,0,k)})_{\cdot jm}  - \sum_{\rho\in\mathcal{R}} (a^{(1,0,k)})_{\cdot j} \cdot \rho_m \Big) e_k \otimes e_j \otimes e_m \nonumber \\[1ex]
=&~2 a^{(2,0,\cdot)}.
\end{align*}
As for the third order term (quadrupole moment), adding \eqref{a-b-3} into \eqref{eq:b-I-3} and after a direct algebraic manipulation, one can acquire that
\begin{align}
\mathcal{I}_3[\bar{u}] = -6 a^{(3,0,\cdot)}.
\end{align}

Hence, we obtain the practical formulation of continuous coefficients for $i=1,2,3$ via force moments 
\begin{align}
    a^{(1,0,\cdot)} &= -\mathcal{I}_1[\bar{u}], \quad a^{(1,1,\cdot)} = -\mathcal{I}_1[\bar{u}], \nonumber \\[1ex]
    a^{(2,0,\cdot)} &= \frac{1}{2}\mathcal{I}_2[\bar{u}], \quad 
    a^{(3,0,\cdot)} = -\frac{1}{6}\mathcal{I}_3[\bar{u}],
\end{align}
which can be applied directly in our main algorithm (cf. Algorithm~\ref{alg:moment_iter_a}).

\subsection{The computation of continuous Green's functions}
\label{sec:sub:gf}

In this section, we provide a brief overview of the implementation of continuous Green's functions. These Green's functions offer several advantages, including the ability to deduce regularity, decay, and even homogeneity \cite{2021-defectexpansion}. Importantly, the representation is computationally promising, as it only requires a finite surface integral in Fourier space for evaluation.

The continuous Green's functions and their higher order correction can be expressed by applying the Morrey formula (cf.~\cite[Section 6.4]{2021-defectexpansion}), that is, for $n\geq1$ and $\ell \neq 0$ 
\begin{equation} \label{eq:morreykernel}
G_{n}(\ell) = (-\Delta)^P \frac{c_{\rm vol}}{(2 \pi)^d}  \int_{\mathbb{S}^{d-1}} \mathcal{A}_{2n-2}(\sigma) J_{-1-h}(\ell \cdot \sigma) \dsigma,
\end{equation}
where $P=\lceil \frac{d+2n-1}{2}\rceil$, $\Delta$ is the Laplacian, and
\begin{align*}
J_t(w) &= t! (-iw)^{-t-1}, &\text{ for } t \geq 0\\[1ex]
J_t(w) &= \frac{1}{(-t-1)!} (iw)^{-t-1}\Big(-\log(-iw) + \sum_{j=1}^{-t-1} j^{-1}\Big), &\text{ for } t < 0,
\end{align*}
where $J_t'(w)=iJ_{t+1}(w)$ is satisfied. We only consider $d=3$ throughout this paper and the results for $d=2$ could be obtained similarly. In particular, $G_0$ reads
\begin{equation}\label{eq:g0_}
G_{0}(\ell) = (-\Delta) \frac{c_{\rm vol}}{(2 \pi)^3}  \int_{\mathbb{S}^{2}} \mathcal{A}_{-2}(\sigma) J_{-2}(\ell \cdot \sigma) \dsigma.
\end{equation}
Note that $\mathcal{A}_{-2} = \widehat{H}_2^{-1}$ where $\widehat{H}_{2n}^{-1}$, the inverse of the discrete lattice operator Fourier multiplier~\cite{braun2016existence, hudson2012stability}, is explicitly defined in~\cite[Section 6.2]{2021-defectexpansion}. For real $w$ we have the identity
\[{\rm Re} J_{-2}(w) =  (-iw) i {\rm Arg}(-iw) = - \pi/2 \lvert w \rvert.\]
Hence, one can further write \eqref{eq:g0_} as 
\begin{equation}
G_{0}(\ell) = \frac{c_{\rm vol}}{16 \pi^2}  \Delta  \int_{\mathbb{S}^{2}} \mathcal{A}_{-2}(\sigma) \lvert \ell \cdot \sigma \rvert \dsigma.
\end{equation}

As discussed in \cite[Section 6]{2021-defectexpansion}, the integral defines a $C^\infty$ function. One can easily evaluate the Laplacian in the distributional sense. We thus obtain
\begin{equation}\label{eq:G0}
G_{0}(\ell) = \frac{c_{\rm vol}}{8 \pi^2}  \int_{\mathbb{S}^{2} \cap \{\sigma \cdot x =0\}} \mathcal{A}_{-2}(\sigma) \dsigma,
\end{equation}
which yields a rigorous derivation of Barnett's formula~\cite{barnett1972precise}.

Applying \eqref{eq:morreykernel} as well as the similar calculation as that on $G_0(\ell)$, the first order correction $G_1(\ell)$ can be rewritten as 
\begin{align}\label{eq:G1}
  G_{1}(\ell) &= (\Delta)^2 \frac{c_{\rm vol}}{(2 \pi)^3}  \int_{\mathbb{S}^{2}} \mathcal{A}_{0}(\sigma) J_{-2}(\ell \cdot \sigma) \dsigma \nonumber \\[1ex]
  &= \frac{c_{\rm vol}}{8 \pi^2} \Delta  \int_{\mathbb{S}^{2} \cap \{\sigma \cdot x =0\}} \mathcal{A}_{0}(\sigma) \dsigma,
\end{align}
where $\mathcal{A}_0 := -  \widehat{H}_2^{-1} \widehat{H}_4 \widehat{H}_2^{-1}$. 

In practical implementation, we make full use of \eqref{eq:G0} and \eqref{eq:G1} to obtain the continuous Green's function $G_0$ and its first-order correction $G_1$. As discussed in the main context, the higher-order derivatives of $G_0$ are obtained using the automatic differentiation~\cite{revels2016forward}.

\appendix
\renewcommand\thesection{\appendixname~\Alph{section}}

\bibliographystyle{plain}
\bibliography{bib}

\end{document}